\documentclass[11pt]{amsart}

\usepackage{a4wide}
\usepackage[title,titletoc,toc]{appendix}
\usepackage[T1]{fontenc}
\usepackage{amsmath}
\usepackage{amssymb}
\usepackage{amsfonts}
\usepackage[english]{babel}
\usepackage{amsthm}
\usepackage{enumerate}
\usepackage{bbm}
\usepackage{pdfsync}
\usepackage{graphicx,color}
\usepackage[all]{xy}
\usepackage{cleveref}
\usepackage{epsfig}
\graphicspath{{./}}
\usepackage{epstopdf}
\ifpdf
  \DeclareGraphicsExtensions{.eps,.pdf,.png,.jpg,.ps}
\else
  \DeclareGraphicsExtensions{.eps}
\fi

\usepackage[top=3cm, bottom=2cm, left=3cm, right=3cm]{geometry}
\usepackage[foot]{amsaddr}

\usepackage{mathrsfs}

\theoremstyle{plain}
\newtheorem{theorem}{Theorem}[section]
\newtheorem{proposition}[theorem]{Proposition}
\newtheorem{lemma}[theorem]{Lemma}
\newtheorem{corollary}[theorem]{Corollary}

\theoremstyle{definition}
\newtheorem{definition}[theorem]{Definition}

\newtheorem{remark}[theorem]{Remark}

\numberwithin{equation}{section}

\newcommand{\R}{\mathbb{R}}

\newcommand{\eps}{\varepsilon}

\newcommand{\mY}{{\mathcal Y}}

\newcommand{\bC}{{\mathbbm{1}}}

\newcommand{\textin}{\textrm{ in }}
\newcommand{\p}{\partial}
\newcommand{\Ds}{(-\Delta)^{s}}

\newcommand{\norm}[2][]{\left\|{#2}\right\|_{#1}}
\newcommand{\angles}[1]{\left\langle{#1}\right\rangle}
\DeclareMathOperator{\dist}{dist}
\DeclareMathOperator{\loc}{loc}

\newcommand{\set}[1]{\left\{#1\right\}}

\title{Uniqueness of entire ground states for the fractional plasma problem}

\author{Hardy Chan}
\author{Mar\'ia Del Mar Gonz\'alez }
\author{Yanghong Huang}
\author{Edoardo Mainini}
\author{Bruno Volzone}

\newcommand{\Addresses}{{
  \bigskip
  \footnotesize

  H. Chan, \textsc{ETH Z\"{u}rich, Departement Mathematik. R\"{a}mistrasse 101, 8092 Z\"{u}rich. Switzerland.}\par\nopagebreak
  \textit{E-mail address}: \texttt{hardy.chan@math.ethz.ch}

  \medskip

  M. d. M. Gonz\'alez, \textsc{Universidad Aut\'onoma de Madrid, Departamento de Matem\'aticas, and ICMAT. Madrid 28049. Spain.}\par\nopagebreak
  \textit{E-mail address}: \texttt{mariamar.gonzalezn@uam.es}

  \medskip

  Y.~Huang, \textsc{Department of Mathematics, University of Manchester,
	Oxford Rd,	Manchester,	M13 9PL UK}\par\nopagebreak
  \textit{E-mail address}: \texttt{yanghong.huang@manchester.ac.uk}

  \medskip

  E.~Mainini, \textsc{Universit\`a  degli studi di Genova, Dipartimento di Ingegneria meccanica, energetica, gestionale e dei trasporti,
   Via all'Opera Pia 15 - 16145 Genova, Italy.}\par\nopagebreak
  \textit{E-mail address}: \texttt{mainini@dime.unige.it}

\medskip

  B.~Volzone, \textsc{Universit\`{a} degli Studi di
Napoli ``Parthenope'', Dipartimento di Scienze e Tecnologie, Centro
Direzionale Isola C4, 80143 Napoli, Italy.}\par\nopagebreak
  \textit{E-mail address}: \texttt{bruno.volzone@uniparthenope.it}

}}


\begin{document}

\begin{abstract} We establish uniqueness of vanishing radially decreasing entire solutions, which we call \emph{ground states}, to some semilinear fractional elliptic equations. In particular, we treat the fractional plasma equation and the supercritical power nonlinearity.
As an application, we deduce uniqueness of radial steady states for nonlocal aggregation-diffusion equations of Keller-Segel type, even in the regime that is dominated by aggregation.
\end{abstract}

\maketitle



\section{Introduction }
We study positive entire \emph{ground states} to the fractional semilinear equation
\begin{equation}
\label{mainequation}
(-\Delta)^s u=a(u-\mathcal{C})_+^p
\quad\mbox{in}\;\;\mathbb R^N,
\end{equation}
where the parameters are in the range
$$
0<s<1,\qquad p\ge 1,\qquad \mathcal C\ge 0,\qquad a>0.
$$
 Here $(-\Delta)^s$ is the fractional Laplace operator on $\mathbb R^N$ ($s<1/2$ if $N=1$).
Moreover $x_+:=0\vee x$ denotes the maximum of $0$ and $x$.
By a {\it ground state}  we mean a bounded positive solution $u$ to \eqref{mainequation} which is \emph{radially decreasing} and decays at infinity, i.e.,
$
u(x)\rightarrow 0\ \text{for }|x|\rightarrow\infty.
$

In the  {\it subcritical case} $p<(N+2s)/(N-2s)$ with $\mathcal C>0$,  the free boundary problem \eqref{mainequation} is the so called \emph{fractional plasma equation}, and it is the object of our first main result.

\begin{theorem}\label{maintheorem}
Let $1\le p<{(N+2s)}/{(N-2s)}$ and $\mathcal C>0$.
There exists a unique ground state for equation \eqref{mainequation}.
\end{theorem}

In our second main theorem, we investigate ground states in the \emph{critical}
and \emph{supercritical regime} $p\geq (N+2s)/(N-2s)$ to equation \eqref{mainequation}, with the choice $\mathcal{C}=0$.
A nontrivial solution exists only for  this special case, as we will show  that there are no ground states if $\mathcal C>0$ and $p\geq (N+2s)/(N-2s)$.

\begin{theorem}\label{th:super}
Let $p\ge{(N+2s)}/{(N-2s)}$. Let $\mathcal C=0$ and $b>0$.
There exists a unique ground state $u$ for equation \eqref{mainequation} such that $u(0)=b$.
\end{theorem}

In the above results, ground state  solutions are interpreted in the distributional sense. However, these solutions turn out to be continuous (hence smooth) and the equation is also satisfied pointwise everywhere in $\mathbb R^N$. Moreover, in the subcritical case covered by Theorem \ref{maintheorem}, the solution is also a weak energy solution, i.e., it belongs to the natural energy space $\dot{H}^{s}(\R^{N})$,
which is a fractional homogeneous Sobolev space. Precise definitions are addressed in Section \ref{prel}.

The construction of ground state solutions (by means of critical point theory) for more general subcritical nonlinearities   than \eqref{mainequation} is found in \cite{Ikoma}. 
On the other hand, existence of ground states for the equation $\Ds u=u^p$ with $p\ge (N+2s)/(N-2s)$ is shown in{ ~\cite[Section 6]{ZLO} and~\cite{acgw} along with a precise decay rate (the case $s=1$ is contained in~\cite[Theorem 9.1]{Souplet}).}  Therefore, our main contribution here is the proof of uniqueness.

We also remark that Theorem \ref{maintheorem} holds true for $0<p <1$ as well, as a consequence of the results proved in \cite{CCH3}, \cite{CHMV} and \cite{DYY}   in the equivalent context of Euler-Lagrange equations associated to aggregation-diffusion free energies that we shall describe in detail through the paper. However, the methods in the proof of Theorem \ref{maintheorem} cannot be applied in case $0<p<1$ since they strongly rely on convexity.
 \\[5pt]

\subsection*{The plasma problem}
In the \emph{local}{ setting} ({i.e.},  $s=1$), the subcritical regime corresponds to  $1\leq p<(N+2)/(N-2)$ for $N\geq3$ ($p\geq 1 $ for $N=2$). In this framework,  equation \eqref{mainequation} with $\mathcal C>0$,  posed in a bounded domain $\Omega$ with homogeneous Dirichlet boundary conditions,  is  the so-called  \emph{plasma problem}.  This particular free boundary problem was introduced in \cite{Temam1} and \cite{Temam2}.
The two-dimensional case was solved in  \cite{BandFluch} and the case $p=1$ in \cite{CaffFried}, while the higher dimensional case was studied in detail in \cite{Flucher-Wei}.  The two-dimensional problem has  an interpretation in plasma physics, because in this context the domain $\Omega$ represents the cross section of a Tokamak machine, a toroidal shell containing a plasma ring surrounded by vacuum. The equations from the magnetohydrodynamics plus further equations modeling the physical properties of the plasma lead to the homogeneous Dirichlet problem for the equation
\begin{equation}\label{epsilonplasma}
-\varepsilon^{2}\Delta u=(u-\mathcal{C})_{+}^{p},
\end{equation}
with a small parameter $\varepsilon$. This equation is equivalent to a nonlinear eigenvalue problem
\[
-\Delta u=\lambda (u-\mathcal{C})_{+}^{p}
\]
where the region inhabited by the plasma is exactly the set $\left\{x \in \Omega \mid u(x)>\mathcal C\right\}$, with $u$ modeling the flux function.
For such model in the form \eqref{epsilonplasma},  the existence of a unique radial ground state is shown to be essential for the characterization of the critical points of least energy solutions $u_{\varepsilon}$ (see \cite{Flucher-Wei}).

In the \emph{nonlocal} setting $s\in(0,1)$, the  Dirichlet problem
\begin{equation*}\label{sol.1}
 \begin{cases}
A_{s} u=\lambda (u-\mathcal{C})^{p}_{+}\,
  & \text{in}\,\Omega, \\
  u=0\,
  &\text{on}\,\partial \Omega
 \end{cases}
\end{equation*}
with the \emph{spectral fractional Laplacian} $A_{s}$ was firstly investigated in \cite{Allen} for $p=1$. In particular the author in \cite{Allen} studies  existence and regularity of solutions, and  the nonlocal counterpart of the geometry of the free boundary $\partial\left\{u=\mathcal{C}\right\}$, which was previously obtained in \cite{KindSpruck}. Recently, in \cite{Mugnai} some interesting existence results are established by critical point theory for the eigenvalue problem related to a general nonlocal operator $\mathcal{L}_{K}$ with a singular kernel $K$ (note that $\mathcal{L}_{K}=(-\Delta)^{s}$ for the choice $K(x)=|x|^{-N-2s}$), \emph{i.e.} the problem

\begin{equation*}\label{eigenvMug}
 \begin{cases}
\mathcal{L}_{K} u=\lambda (u-\mathcal{C})^{p}_{+}
\,
  & \text{in}\,\Omega, \\
  u=0\,
  &\text{on}\,\R^{N}\setminus \Omega.
 \end{cases}
\end{equation*}
Then, a motivation for the study of radial ground states for equation \eqref{mainequation}  would rely on the geometric characterization of least energy solutions to the equation
\[
\varepsilon^{2s}(-\Delta)^{s}u=(u-\mathcal{C})_{+}^{p}.
\]

Going back to the local setting $s=1$, a construction of the unique entire ground state for
\begin{equation*}
-\Delta u=(u-\mathcal{C})_+^p\qquad\mbox{in $\mathbb R^N$}\label{equlocal},
\end{equation*}
with $1<p<(N+2)/(N-2)$, $N\ge 3$,
is contained in the paper by Flucher and Wei \cite[Lemma 5]{Flucher-Wei}.
Indeed,  if we put for instance $\mathcal{C}=1$, the construction of \cite{Flucher-Wei} is based on the radiality of the solution together with a simple scaling ODE argument, which gives the following direct representation
\begin{equation*}
u(r)=
 \begin{cases}
1+R^{\frac{2}{1-p}}v(\frac{r}{R})\,
  & r<R,\\[8pt]
  \left(\frac{r}{R}\right)^{2-N}\,
  & r>R.
 \end{cases}
\end{equation*}
Here, $R$ is the radius of the ball $B_{R}=\left\{u>1\right\}$, which is the (unknown) \emph{free boundary} of the problem, and $v$ the unique positive solution in the unit ball $B_{1}$ to the subcritical Dirichlet problem
\begin{equation}\label{vproblem}
 \begin{cases}
-\Delta v=v^{p}\,
  & \text{ in }\,B_{1},\\[8pt]
  v=0\,
  & \text{ on }\,\partial B_{1}.\,
 \end{cases}
\end{equation}
The regularity of the solution $u$ up to the boundary provides also an explicit representation of the radius, i.e., $R=(|v^{\prime}(1)|/(N-2))^{(p-1)/2}$, which is \emph{independent} on the solution $u$ itself.
Notice that since $v$ is radial, the equation in \eqref{vproblem} becomes an ODE, and the smoothness of $v$ up to the boundary (see for instance \cite[Theorem 8.29, Theorem 6.19]{Gilbarg}) forces one to have the condition
\[
\frac{v^{\prime\prime}(1)}{|v^{\prime}(1)|}=N-1,
\]
which yields in particular that $u\in C^{2}$ 
(actually, at least $u\in C^{2,\alpha}$ for all $\alpha<1$, by elliptic regularity).

We also mention that,  for the case $0<p<1$, the existence-uniqueness result  for such a problem is contained in \cite[Theorem 1]{BreOsw}; moreover, this solution is radial due to the rotational invariance of the operator. The particular case $p=1$ is more explicit, since, imposing the continuity of the radial derivative we have that $u$ has the following expression
\begin{equation*}
u(r)=
 \begin{cases}
\left(1+\frac{R^{N/2}(N-2)}{\mathcal{J}_{N/2}(R)}r^{-(N-2)/2}\,\mathcal{J}_{(N-2)/2}(r)\right)\,
  & r<R,\\[8pt]
  \left(\dfrac{r}{R}\right)^{2-N}\,
  & r>R\,\\[8pt]
 \end{cases}
\end{equation*}
where $R=z_{0}$ is the first zero of the Bessel function of the first kind $\mathcal{J}_{(N-2)/2}$.\\

In the nonlocal setting $s\in (0,1)$ this kind of local ODE approach is no longer available, so any
attempt to achieve an explicit representation of the ground states is out of sight. Instead, the
techniques that we shall use in the proofs of the uniqueness result in Theorem \ref{maintheorem}
(and also Theorem  \ref{th:super}) rely on the applications of a monotonicity formula developed for
the fractional Schr\"{o}dinger equation by Frank, Lenzmann and Silvestre in \cite[Theorem
2.1]{FLeSil}, inspired by the work of Cabr\'e and Sire \cite{Cabre-Sire}. In particular, we will
work with the equation satisfied by the difference of two solutions $u_1,u_2$, written in terms of a
potential term of the form $\mathcal V(r):=\frac{(u_1)^p_+-(u_2)^p_+}{u_1-u_2}$. Surprisingly enough, the monotonicity argument still works here since the potential can
be shown to be decreasing even though we do not know the location of the free boundaries $R_1$,
$R_2$. In addition, the scaling properties of \eqref{mainequation} will be essential to uniquely identify the central density of the solutions and get the final uniqueness result.\\[5pt]

\subsection*{Steady states of aggregation-diffusion equations}

An application of our main results, that we extensively develop through the paper (see Section \ref{uniqsteady}),  concerns the analysis of steady states for the following fractional aggregation-diffusion equation
\begin{equation}\label{eq:KS}
 \partial_t \rho = \Delta \rho^m -	\chi\nabla \cdot \left( \rho \, \nabla (-\Delta)^{-s}\rho\right)
\end{equation}
for a density $\rho(t,x)$ defined on $\R_+ \times \R^N$.  Here, $\chi>0$ is a constant, $m>1$ is the diffusion parameter, and $(-\Delta)^{-s}\rho$ is the Riesz potential of $\rho$, namely
the convolution of $\rho$ with the Riesz kernel
 $c_{N,s}|x|^{2s-N}$,  
where the normalization constant $c_{N,s}$ is given by
\[
c_{N,s}= \frac{\Gamma\left(\frac{N}{2}-s\right)}{\pi^{N/2}4^s\Gamma(s)}. \]
It is shown in \cite{CHMV} that in the  {\it diffusion-dominated regime}, namely $m>m_c:=2-\frac{2s}N$, steady states  for the dynamics \eqref{eq:KS} are characterized as nonnegative radially decreasing solutions to the Euler-Lagrange equation
\begin{equation}\label{steadyintro}
\rho=\left(\tfrac{m-1}{m}\right)^{\frac1{m-1}}\left(\chi(-\Delta)^{-s}\rho-\mathcal{K}\right)_+^{\frac{1}{m-1}},
\end{equation}
where $\mathcal{K}$ is a  positive constant (playing the role of a Lagrange multiplier). 
 Then, the Riesz potential of a solution $\rho$ to the above equation, namely $u:=(-\Delta)^{-s}\rho$, formally satisfies equation \eqref{mainequation} with $p=\tfrac 1{m-1}$,  $a=((m-1)\chi/m)^{1/(m-1)}$ and $\mathcal C=\mathcal K/\chi$.
The application of our result will be therefore a proof of  the uniqueness of radial steady states of equation \eqref{eq:KS}.
We stress that the  diffusion-dominated regime is found in the subcritical range as it corresponds to $p<p_c:={N}/{(N-2s)}$, see Figure \ref{f1}. On the other hand, we may treat the case $p\ge{N}/{(N-2s)}$ as well, thus obtaining a characterization of the radial stationary states even in the so-called {\it aggregation-dominated regime}.

\begin{figure}[h]
\begin{center}
\includegraphics{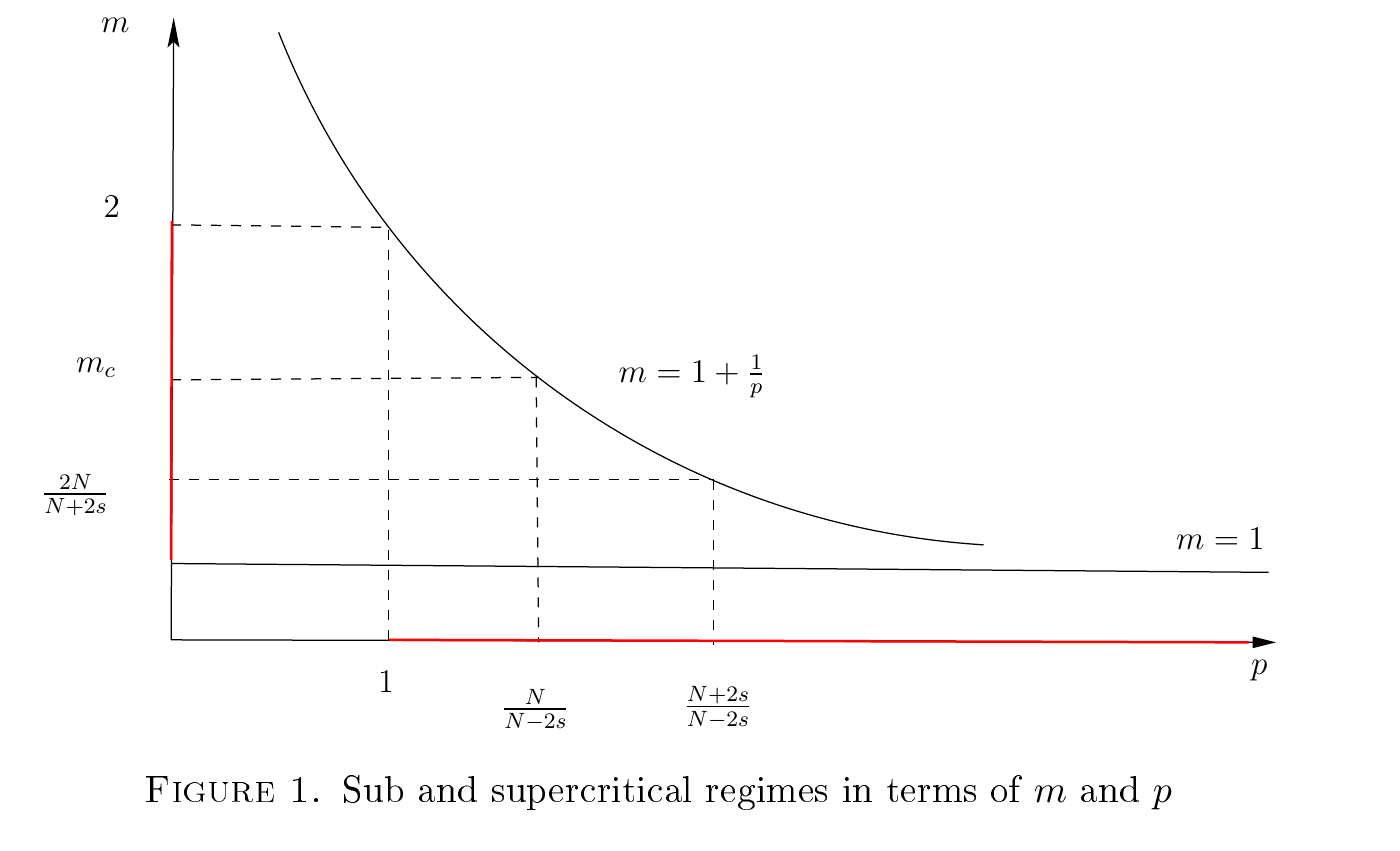}\label{f1}
\end{center}
\end{figure}



In the case $m>m_c$, our results about stationary states complement the ones in \cite{CHMV}, where  their  regularity properties are established in detail. In this regime, uniqueness (up to translations) of radial stationary states with given mass $M:=\int_{\mathbb R^N}\rho>0$ can be deduced by the result in   \cite{CCH3}, see also \cite{DYY} for analogous results in the range $m\ge 2$. {In this regard, in Section \ref{uniqsteady} we give an alternative proof of the uniqueness of the stationary states in the range $m\in(m_c,2]$, by applying Theorem \ref{maintheorem}}.

 We stress that in the {diffusion-dominated regime}, stationary states turn out to be minimizers of the free energy functional associated to the dynamics, i.e.,
%
\begin{equation}\label{functional}
\mathcal{F}[\rho]=\frac{1}{m-1}\int_{\mathbb{R}^N}\rho^m(x)\,dx-\frac\chi2\int_{\mathbb{R}^N}\int_{\mathbb{R}^N}c_{N,s}|x-y|^{2s-N}\rho(x)\rho(y)\,dx\,dy,
\end{equation}
among densities $\rho\in L^m_+(\mathbb R^N)$ with prescribed mass $M>0$. In Section \ref{mass scaling}, we shall further investigate the behavior of stationary states as a function of the mass $M$.
Indeed, we shall remark that two stationary states of different masses are rescalings of one another, and the value of the Lagrange multiplier $\mathcal{K}$ in the equation \eqref{steadyintro} is uniquely determined by the mass. In fact, $M$ and $\mathcal{K}$ are related by a bijection of $(0,+\infty)$ onto itself, so that the set of stationary states is a one-parameter family, where the parameter can be chosen to be either $M$ or $\mathcal{K}$. In the local setting $s=1$, the classical results by Lieb and Yau \cite{LY} provide a complete description of the properties of the family of minimizers, by investigating the relations between the mass and other relevant quantities such as the central density $\rho(0)$ or the radius $R$ of the support.   Our results in Section \ref{mass scaling} provide the same information in the fractional case,
along with a precise scaling exponent $\ell=\ell (m,s,N):={\frac{(m-2)N+2sm}{(m-2)N+2s}}$ of the minimal value of $\mathcal F$ as a function of $M$ within the family of minimizers; see Lemma \ref{basicestimates} and Theorem \ref{scaletheorem}.
 These results are only based on the uniqueness of minimizers of given mass and they extend therefore  to the regime $m>2$ (i.e., $0<p<1$), where uniqueness is given by \cite{CCH3, DYY}, even if our main uniqueness theorem does not apply for such values of $m$.\\

In the case  $m=m_c$ \emph{i.e.} the so-called \emph{fair competition regime}, there is a degeneracy in the behavior of the mass $M$ in the family of stationary states, which can be seen from the degeneracy of the above exponent $\ell$. In this regime there exists indeed a critical mass $M_c$ such that all stationary states have mass $M_c$.
In this case, our uniqueness result from Theorem \ref{maintheorem} can be used to conclude that stationary states still form a one-parameter family. As a parameter one may take the Lagrange multiplier $\mathcal{K}\in(0,+\infty)$.
The value $M_c$ is related to the optimal constant in a suitable version of the Hardy-Littlewood-Sobolev inequality as proved in \cite{CCH}. In fact, $M_c$ turns out to be the only value of the mass for which functional $\mathcal F$ has minimizers (in this case a one-parameter family of minimizers). We also refer to \cite{CCH, CCH2} for a detailed analysis of the fair competition regime.\\ 

Eventually, our uniqueness results can be applied in the {aggregation-dominated regime} $m\in (1,m_c)$ to yield a novel characterization of stationary states, as we shall detail in Section \ref{uniqsteady}. There are three subcases where different behaviors occur.
If $m\in(\tfrac{2N}{N+2s},m_c)$, solutions to \eqref{steadyintro} with finite mass and positive Lagrange multiplies do exist, thus providing a natural notion of stationary state even if in this case there are no minimizers of the functional $\mathcal F$ anymore. Again, there is a one-parameter family of stationary states, parameterized by the mass. 
In case $m\in (1,\tfrac{2N}{N+2s}]$, we will show that no radially decreasing solutions to equation \eqref{steadyintro} exist if $\mathcal{K}>0$. In this setting, we must have $\mathcal{K}=0$ and stationary states are not compactly supported anymore. Instead they are smooth functions, slowly decaying at infinity (with a precise decay rate) for $m\in (1,\tfrac{2N}{N+2s})$.
The value $m=\tfrac{2N}{N+2s}$ corresponds the the critical exponent $p=\tfrac{N+2s}{N-2s}$ in \eqref{mainequation}. The result by Chen, Li and Ou \cite{CLO} provides a complete, explicit description of the one-parameter family of stationary states in this case.
In case $m\in(1, \tfrac{2N}{N+2s})$, thanks to our uniqueness result from Theorem \ref{th:super} we obtain once again a one-parameter family of stationary states.  However, these steady states have infinite mass and the family can be parameterized by the value of central density $\rho(0)$. 
\\

We also address the reader to the paper of Bian and Liu \cite{BL}, where an analogous full investigation of stationary states in the different regimes is provided  for the local case $s=1$. The different thresholds are found by formally putting $s=1$ in our setting: radial stationary states  are compactly supported  for $m>\tfrac{2N}{N+2}$, while they are supported on the whole of $\mathbb R^N$ if $m\le\tfrac{2N}{N+2}$, and they are  explicit for $m=\tfrac{2N}{N+2}$.\\[5pt]

\subsection*{Numerical approximation of the fractional plasma equation} In Section~\ref{numerical},
a numerical method is  proposed for~\eqref{mainequation} with $\mathcal{C}>0$ and $p < (N+2s)/(N-2s)$, also covering the case $0<p<1$, by taking advantage of the fact that $\rho = (-\Delta)^s u$ is supported on a ball and can hence be  expanded using appropriate Jacobi polynomials in the radial variable.
These special types of  Jacobi polynomials are chosen because the Riesz potential $u = (-\Delta)^{-s}\rho$ can be easily evaluated,
by extending some explicit relations from~\cite{MR3640641}.
As a result, the main equation~\eqref{mainequation} is reduced to a system of algebraic equations for the expansion coefficients,
 subsequently solved by a fixed point iteration for $p < 1$ or standard Newton's method for nonlinear equations for general $p$. The solutions
as either $s$ or $p$ varies are illustrated in different figures, showing the dependence of their behaviors on these two parameters.
Besides providing quantitative examples to further explore analytical properties of the solutions to \eqref{mainequation},
this method can also be used to approximate radial steady states of the aggregation-diffusion
equation~\eqref{eq:KS}. These steady solutions are usually obtained by finding the numerical steady states at large time, with algorithms for instance as the one in~\cite{MR3372289}, based on the gradient flow structure of the
evolution equation and on special techniques to preserve the nonnegativity of the solution. The method proposed in this
paper employs more efficient iterative solver, while avoiding complicated calculations of functions in the radial variable.
\\[5pt]

 \subsection*{ Organization of the paper} In Section \ref{prel} we give some basic definitions concerning the essential functional framework. Furthermore, several existence results and regularity properties of solutions will be introduced. Section \ref{uniqsubcase} and Section \ref{critsupercrreg} are entirely devoted to the proofs of Theorem \ref{maintheorem} and Theorem \ref{th:super}, respectively. In Section \ref{uniqsteady} we provide our main applications of the above-cited results, that is the uniqueness properties of the steady states to the Keller-Segel evolution equation \eqref{eq:KS}.
 Section \ref{mass scaling} provides further investigation of steady states of \eqref{eq:KS}, in the diffusion dominated regime, by focusing on their scaling properties with respect to the mass of the density $\rho$.
 Section \ref{numerical} exploits certain numerical aspects of the ground states for \eqref{mainequation} in the subcritical case, and an algorithm is proposed for a numerical solution.




%

\section{Preliminaries: functional background and regularity properties of solutions}\label{prel}

\subsection{The fractional Laplacian and the extension problem}
Let $s\in(0,1)$, with $s<1/2$ if $N=1$.
The fractional Laplacian  $\Ds$ on $\mathbb R^N$ is defined by means of Fourier transform as $$\widehat{\Ds u}(\xi)=|\xi|^{2s}\hat u(\xi).$$ For smooth enough $u$ (see \cite[Proposition 2.4]{Silvestre:regularity}), it can be calculated pointwise as the singular integral
\begin{equation}\label{singularintegral}
(-\Delta)^s u(x)=C_{N,s}P.V. \int_{\mathbb R^N} \frac{u(x)-u(y)}{|x-y|^{N+2s}}\,dy,
\end{equation}
where $C_{N,s}$ is an explicit normalization constant, given by
\[
C^{-1}_{N,s}=\int_{\R^{N}}\frac{1-\cos(y_{1})}{|y|^{N+2s}}\,dy.
\]
The fractional Gagliardo seminorm is defined
\begin{equation*}
[u]_{\dot H^s(\mathbb R^N)}=\left(\frac{C_{N,s}}{2}\int_{\mathbb R^N}\int_{\mathbb R^N} \frac{|u(x)-u(y)|^2}{|x-y|^{N+2s}}\,dxdy\right)^{1/2},
\end{equation*}
and the homogeneous Sobolev space $\dot H^s(\mathbb R^N)$ is the completion of $C^\infty_c(\mathbb R^N)$ with respect to $[\cdot]_{\dot H^s(\mathbb R^N)}$. Actually (see Chapter 15 in \cite{Ponce:book} and the references therein),
\[
\dot H^{s}(\R^{N})=\left\{u\in L^{2^{\ast}_{s}}(\R^{N}):\,[u]_{\dot H^{s}(\R^{N})}<\infty \right\},
\]
where we have defined, as customary,
 $$2^{*}_{s}:=\dfrac{2N}{N-2s}.$$

For $u=u(x)$, we consider the  $s$-harmonic (or Poisson) extension  $U=U(x,y)$ on the upper-half space $\R^{N+1}_{+}=\{(x,y)\,:\,x\in\mathbb R^N,y>0\}$, which the solution of the Dirichlet problem
\begin{equation}\label{extension}
\begin{cases}
\Delta_x U+\frac{1-2s}{y}\partial_y U+\partial_{yy}U=0,
	&\text{ in }\R^{N+1}_{+},\\
U(x,0)=u(x),&x\in\R^{N}.
\end{cases}
\end{equation}
Such $U$ is given by the explicit formula
\begin{equation}\label{Poisson-extension}
U(x,y)=c\int_{\mathbb R^N} \frac{y^{2s}}{\left(|x-\zeta|^2+y^2\right)^{\frac{N+2s}{2}}}u(\zeta)\,d\zeta,
\end{equation}
where $$c=\left(\int_{\R^N}\frac{d\zeta}{(|\zeta|^2+1)^{\frac{N+2s}{2}}}\right)^{-1}$$ is an explicit normalization constant.
In addition, it is well known (\cite{CaffSilv}) that
\begin{equation*}\label{DtN}
(-\Delta)^s u=-d_{s} \lim_{y\to 0} y^{1-2s}\partial_y U=:D_{s}U.
\end{equation*}
Here we have defined the constant
$$
d_s:=\frac{2^{2s-1}\Gamma(s)}{\Gamma(1-s)}.
$$
We further introduce the homogeneous weighted Sobolev space $\dot{\mathcal{H}}^{1}(\R^{N+1}_{+},y^{1-2s})$, which is defined as the completion of $C_{c}^{\infty}(\overline{\R^{N+1}})$ with respect to the norm
\[
\|U\|_{\dot{\mathcal{H}}^{1}(\R^{N+1}_{+},y^{1-2s})}=\int_{\R^{N+1}_{+}}y^{1-2s}|\nabla_{x,y}U|^{2}\,
dx\,dy.
\]


\subsection{Several definitions of weak solutions}
We always assume $a>0$, $p\ge 1$ and $\mathcal C\ge 0$.
We introduce two notions of weak solutions for problem \eqref{mainequation}.
 We first define {\it weak energy solutions} according to the following:

\begin{definition}[Weak energy solution]\label{weakenergy} Let $p\geq1$.
We say that a function $u\in \dot H^{s}(\R^{N})\cap L^{\infty}(\R^{N})$ is a weak energy solution to \eqref{mainequation} if
\begin{equation}\label{energyequation}
\frac{C_{N,s}}{2}\int_{\R^{2N}}\frac{(u(x)-u(y))(\phi(x)-\phi(y))}{|x-y|^{N+2s}}\,dx\,dy
=\int_{\R^{N}}a(u-\mathcal{C})_+^p\,\phi\,dx\qquad \forall \phi\in C_c^{\infty}(\R^{N}).
\end{equation}
Moreover, we say that $U$ is a weak energy solution to \eqref{extension} with the Neumann boundary condition
\begin{equation*}\label{equation-extension}
D_{s}U=a(U(x,0)-\mathcal{C})_{+}^{p}
\end{equation*}
if $U\in\dot{\mathcal H}^{1}(\mathbb R^{N+1}_+,y^{1-2s})$, $U(x,0)\in L^{\infty}(\R^{N})$ and
\begin{equation*}
\int_{\mathbb R^{N+1}_+}y^{1-2s}\nabla U\cdot\nabla \Phi\,dxdy=\int_{\mathbb R^N} \Phi(\cdot,0)a(U(x,0)-\mathcal{C})_+^p\,dx
\end{equation*}
for every $\Phi$ smooth test function compactly supported in $\overline{\mathbb R^{N+1}_+}$.
\end{definition}
\begin{remark}From the previous definition, it follows that if $U$ is a weak energy solution to the extension problem \eqref{extension}, then its trace $u(x):=U(x,0)$ is a weak energy solution to \eqref{mainequation}.
Moreover, we notice that  in the case $p<(N+2s)/(N-2s)$ and $\mathcal{C}>0$,  we have $(u-\mathcal C)_+^p\in L^{1}(\R^{N})$.  \end{remark}

Now we introduce the more general notion of {\it distributional solution}. The importance of Definition \ref{distrib} and Proposition \ref{dec} on distributional solutions will come up especially when considering  the \emph{supercritical} regime $p>(N+2s)/(N-2s)$. Indeed, we will see below in Proposition \ref{$C=00$} that there are distributional solutions that do not belong to  the energy space $\dot H^s(\mathbb R^N)$.
Let us first introduce the weighted space
\begin{equation*}
L^1_s(\mathbb R^N)=\left\{u\in L^1_{\rm loc}(\mathbb R^N) \,:\, \int_{\mathbb R^N} \frac{|u(x)|}{(1+|x|^2)^{\frac{N+2s}{2}}}\,dx<\infty\right\}.
\end{equation*}

\begin{definition}[Distributional solution]\label{distrib}
We say that $u\in L^{\infty}(\mathbb R^N)$  is a distributional solution to \eqref{mainequation} if $u\in L^1_s(\mathbb R^N)$ and \begin{equation}\label{fractional-up}
\int_{\mathbb R^N} u (-\Delta)^s\phi\,dx =\int_{\mathbb R^N} a(u-\mathcal{C})_{+}^p\phi\,dx\qquad\text{for all } \phi\in C^\infty_c(\mathbb R^N).
\end{equation}
\end{definition}
Notice that the previous definition makes sense because of the assumption $u\in L_s^1(\mathbb R^N)$, since for $\phi \in C^{\infty}_c(\mathbb R^N)$, we have $(-\Delta)^s\phi \in \mathcal S_s$, where
\[
\mathcal S_s:=\{f\in C^\infty(\mathbb R^n): (1+|\cdot|^{N+2s})D^\beta f(\cdot)\in L^\infty(\mathbb R^N)\quad\forall \beta\in\mathbb N^n_0 \}.
\]
Here $\mathbb{N}^n_0:=\set{0,1,\dots}^n$ is the ordered $n$-tuples of non-negative integers.

%
The following result simply states that the definition of weak energy solution is stronger than the distributional one.
\begin{proposition}\label{rosa} Let $p\geq1$.
Let $u\in \dot H^s(\mathbb R^N)\cap L^{\infty}(\R^{N})$ be a weak energy solution to \eqref{mainequation}. Then it is also a  distributional solution. 
\end{proposition}
\begin{proof} Since $u$ belongs to the homogeneous space $\dot H^s(\mathbb{R}^N)$, then $u\in L^1_s(\mathbb R^N)$, by Hardy-Littlewood-Sobolev inequality.
Since $u$ is a weak energy solution according to Definition \ref{weakenergy}, we have
\[
\frac{C_{N,s}}{2}\int_{\R^{2N}}\frac{(u(x)-u(y))(\phi(x)-\phi(y))}{|x-y|^{N+2s}}\,dx\,dy
=\int_{\R^{N}}a(u-\mathcal{C})_{+}^p\phi\,dx\quad \text{for all }\phi\in C_{c}^{\infty}(\R^{N}),
\]
where the left hand side is the scalar product in $\dot H^s(\mathbb R^N)$, yielding
\[
\int_{\mathbb R^N}a(u-\mathcal C)_+^p\phi\,dx=\langle u,\phi \rangle_{\dot H^s(\mathbb R^N)}=\langle(-\Delta)^{s/2}u,(-\Delta)^{s/2}\phi\rangle_{L^2(\mathbb R^N)}=\langle u,\Ds\phi\rangle,
\]
where $\langle\cdot,\cdot\rangle$ in the right hand side denotes the duality between ${ \dot H^s(\mathbb R^N)}$ and ${ \dot H^{-s}(\mathbb R^N)}$.
Since $u\in L^1_s(\mathbb R^N)$ and $\Ds\phi\in \mathcal S_s$, we have $u\Ds\phi\in L^1(\mathbb R^N)$ and $\langle u,\Ds\phi\rangle=\int_{\mathbb R^n} u\Ds\phi.$ Thus, $u$ satisfies \eqref{fractional-up}.
\end{proof}


One may give a third notion of weak solutions by means of the integral equation
\begin{equation}\label{integral-equation}
u(x)=\int_{\mathbb R^N} \frac{1}{|x-y|^{N-2s}}a(u(y)-\mathcal C)^p_+\,dy.
\end{equation}
This involves defining the Riesz potential of the right hand side. While this is trivial if the right hand side is compactly supported, justifications are needed if $\mathcal C=0$ and $u$ is positive everywhere and vanishing at infinity.
For a distributional solution $u$ to \eqref{mainequation} we shall see that $u$ is the Riesz potential of $a(u-\mathcal C)_+^p$ in the sense of distributions and also pointwise everywhere.  Let us start with the first fact.

\begin{proposition}\label{dec}
Let $u$ be a positive distributional solution to \eqref{mainequation}  satisfying  $u(x)\rightarrow0$ for $|x|\rightarrow\infty$. Then
\begin{equation*}\label{L-s}
\int_{\mathbb R^N}\frac{(u-\mathcal{C})_{+}^p}{1+|x|^{N-2s}}\,dx<+\infty
\end{equation*}
and
\[
\int_{\mathbb R^N} u\,\varphi\,dx=\int_{\mathbb R^N} a(u-\mathcal C)_{+}^p\,(-\Delta)^{-s}\varphi\,dx\qquad \text{for all } \varphi\in C^\infty_c(\mathbb R^N),
\]
where $(-\Delta)^{-s}\varphi$
is the Riesz potential of $\varphi$.
\end{proposition}
\begin{proof}
We only need to consider the case when the right hand side is not compactly supported, which happens for $\mathcal C=0$ only. Thus assume that $u$ is a positive distributional solution to
\begin{equation*}
(-\Delta)^s u=a u^p\quad\text{in }\mathbb R^N.
\end{equation*}
Lemma 5.4 in \cite{Ao-Gonzalez-Hyder-Wei} immediately yields that
$$u^p\in L^1_{-s}(\mathbb R^N):=\left\{u\in L^1_{\rm loc}(\mathbb R^N) \,:\, \int_{\mathbb R^N} \frac{|u(x)|}{(1+|x|^2)^{\frac{N-2s}{2}}}\,dx<\infty\right\}.$$ Now we can extend the validity of \eqref{fractional-up} to test functions $\phi$ of the form $\phi=(-\Delta)^{-s}\varphi$ with $\varphi\in C^\infty_c(\mathbb R^N)$. Indeed, since we have $u^p\phi\in L^1(\mathbb R^N)$, we can pass to the limit by approximating $\phi$ uniformly on $\mathbb R^N$ with a sequence of smooth compactly supported functions $\phi_n$. These are defined by taking an approximating sequence $\phi_n(x)=\eta_n(x)\phi(x)$, where $\eta_n(x)=\eta(x/n)$ and $\eta$ is a smooth function such that $\eta(x)=1$ if $|x|\le 1$ and $\eta(x)=0$ if $|x|\ge 2$.
\end{proof}


\begin{remark}
In the critical case $p=(N+2s)/(N-2s)$, the fact that the integral equation
\begin{equation*}
u(x)=\int_{\mathbb R^N} \frac{1}{|x-y|^{N-2s}}u(y)^p\,dy
\end{equation*}
is equivalent to the original PDE
\begin{equation*}
(-\Delta)^s u=u^p
\end{equation*}
was shown in \cite{CLO}.\end{remark}
We finally recall the notion of {\it ground state}.
\begin{definition}[Ground state]
We say that a distributional solution $u\in L^{\infty}(\mathbb R^N)$ to \eqref{mainequation} is a ground state for  equation \eqref{mainequation} if it is positive, radially decreasing and vanishing at infinity.
\end{definition}

\bigskip

\subsection{The subcritical case}
Now we provide some considerations concerning the existence of  ground states  in the subcritical case, that is, $1\le p<({N+2s})/({N-2s}).$
We first observe that, in this range, the existence of a nontrivial solution $u$ for equation \eqref{mainequation} forces $\mathcal{C}>0$. Indeed, if $\mathcal{C}=0$, by a Liouville type result contained in \cite{ZLO} we have that $u\equiv0$ is the only  solution to the integral equation \eqref{integral-equation} corresponding to
$(-\Delta)^{s}u=a u^p.$
Then in the subcritical range we will always assume $\mathcal{C}>0$. We start with the following existence result:
\begin{proposition}\label{existence}
Let $p\in (1,(N+2s)/(N-2s))$. Let $\mathcal C>0$. Then there is at least one weak energy solution  to equation \eqref{mainequation}  that is a ground state.
\end{proposition}

\begin{proof}
The existence of a weak energy radially decreasing solution to \eqref{mainequation} could be established by a variation of \cite[Theorem 1.3]{Ikoma} as it can be reached by proving the existence of radial critical points to the energy functional
\[
I(u):=[u]^2_{\dot H^{s}(\R^{N})}-\int_{\R^{N}}F(u)\,dx,
\]
where
\[
F(t):=\int_{0}^{t}f(\tau)\,d\tau=\frac{1}{p+1}\,(t-\mathcal C)^{p+1}_+
\]
is a primitive of $f(t):=(t-\mathcal{C})_{+}^{p}$.
The results in \cite{Ikoma} are given for $N\ge 2$. Here we provide an alternative variational proof of existence that is well-suited to any dimension $N\ge 1$ (recalling that $s<1/2$ if $N=1$).
{Consider the functionals
\[\mathcal G_j(u):=[u]^2_{\dot H^s(\mathbb R^N)}+j\left(\int_{\mathbb{R}^N}F(u)\,dx-1\right)^2,\qquad j\in\mathbb N,\]
defined on $\dot H^s(\mathbb R^N)$. Notice that the continuous embedding of  $\dot H^s(\mathbb R^N)$ into $ L^{2^*_s}(\mathbb R^N)$ yields $F\circ u\in L^1(\mathbb R^N)$ for every $u\in \dot H^s(\mathbb R^N)$, since we are in the regime $p+1<2_s^*$.
For each $j\in\mathbb N$, we claim that $\mathcal G_j$ has a minimizer over $\dot H^s(\mathbb R^N)$ which is a radially decreasing function vanishing at infinity. Indeed, fix $j\in\mathbb N$ and let   $(u_k)_{k\in\mathbb N}\subset \dot H^s(\mathbb R^N)$ be a minimizing sequence for $\mathcal G_j$. Such a sequence is bounded in $\dot H^s(\mathbb R^N)$, since $[u_k]^2_{\dot H^s(\mathbb R^N)}\le\mathcal G_j(u_k)\le 1+\mathcal G_j(w)=1+[w]^2_{\dot H^s(\mathbb R^N)}$ for any $k\in\mathbb N$ large enough as soon as $w\in 	\dot H^s(\mathbb R^N)$ is such that $\int_{\mathbb R^N}F(w)\,dx=1$.  By taking the Schwarz spherical rearrangement, we can assume w.l.o.g. that each $u_k$ is radially decreasing nonnegative and vanishing at infinity, see for instance \cite[Theorem 1.1.1]{Kesavan}.
By the continuous embedding of  $\dot H^s(\mathbb R^N)$ into $ L^{2^*_s}(\mathbb R^N)$, we get the boundedness of the sequence $(u_k)_{k\in\mathbb N}$ in $L^{2^*_s}(\mathbb R^N)$, so that  the measure of the set $\{u_k>\mathcal C\}$ is uniformly bounded with respect to $k$, and since this set is a ball centered at the origin, there exists $r>0$ such that the support of $(u_k-\mathcal C)_+$ is contained in $B_r$ for any $k\in\mathbb N$. The compactness of the embedding $\dot H^s(\mathbb R^N)\hookrightarrow L^{2^*_s}_{loc}(\mathbb R^N)$ shows that up to subsequences $u_k\rightharpoonup u$ weakly in $\dot H^s(\mathbb R^N)$ and $u_k\to u$ strongly in $L^{2^*_s}(B_r)$. Since we have $p+1<2_s^*$ in the subcritical regime, $F(t)$ grows slower than $|t|^{2^*_s}$ at infinity, and then  we deduce
\begin{equation}\label{subcriticality}
\lim_{k\to+\infty}\int_{\mathbb R^N}F(u_k)\,dx=\lim_{k\to+\infty}\int_{B_r}F(u_k)\,dx=\int_{B_r}F(u)\,dx=\int_{\mathbb R^N}F(u)\,dx,
\end{equation}
see for instance the convergence result in \cite[Theorem A.I]{BerLio}.
Along with the weak lower semicontinuity of the $\dot H^s(\mathbb R^N)$ norm, this shows that $u$ is a minimizer of $\mathcal G_j$. Moreover, $u$  is  nonnegative and radially decreasing,
thus the claim is proved.}

{
For every $j\in\mathbb N$, let $u_j$ be a minimizer of $\mathcal G_j$, provided by the latter claim. Let $w$ as above.  By minimality we have the estimate $[u_j]^2_{\dot H^s(\mathbb R^N)}\le \mathcal G_j(u_j)\le \mathcal G_j(w)=[w]^2_{\dot H^s(\mathbb R^N)}$, which shows that the sequence $(\mathcal G_j(u_j))_{j\in\mathbb N}$ is bounded and that the sequence of radially decreasing functions $(u_j)_{j\in\mathbb N}$ is bounded in $\dot H^s(\mathbb R^N)$. Therefore, by subcriticality (i.e., $p+1<2^*_s$)  \cite[Theorem A.I]{BerLio} implies again that up to subsequences $f(u_{j})\rightarrow f(u)$ and $f(u_{j})u_{j}\rightarrow f(u)u$ in $L^{1}_{loc}(\R^{N})$, thus the same arguments used to prove \eqref{subcriticality} yield
\begin{equation}\label{subcriticality2}
\lim_{j\to+\infty}\int_{\mathbb R^N}F(u_j)\,dx=\int_{\mathbb R^N}F(u)\,dx,
\end{equation}
\begin{equation}\label{subvarphi}
\lim_{j\to+\infty}\int_{\mathbb R^N}f(u_j)\,\phi\,dx=\int_{\mathbb R^N}f(u)\,\phi\,dx\qquad\forall\phi\in C^\infty_c(\mathbb R^N)
\end{equation}
and
\begin{equation}\label{subcriticality3}
\lim_{j\to+\infty}\int_{\mathbb R^N}f(u_j)\,u_j\,dx=\int_{\mathbb R^N}f(u)\,u\,dx.
\end{equation}
We may also check that the limit in \eqref{subcriticality3} is strictly positive:  assuming by contradiction that it is zero, we get
\[
\begin{aligned}
\limsup_{j\to+\infty}\int_{\mathbb R^N}F(u_j)\,dx&=\frac1{p+1}\limsup_{j\to+\infty}\int_{\mathbb R^N}(u_j-\mathcal C)_+^{p}(u_j-\mathcal C)_+\,dx\\
&\le\frac1{p+1}\limsup_{j\to+\infty}\int_{\mathbb R^N}f(u_j)\,u_j\,dx=0,
\end{aligned}
\]
therefore \begin{equation}\label{subcriticality4}\lim_{j\to+\infty}\left(\int_{\mathbb R_N}F(u_j)\,dx-1\right)^2=1,\end{equation}
which contradicts the boundedness of the sequence $(\mathcal G_j(u_j))_{j\in\mathbb N}$.
Now,
a first variation argument readily entails that
\begin{equation}\label{eulerlagrange}
\langle u_j,\phi\rangle_{\dot H^s(\mathbb R^N)}+j\left(\int_{\mathbb R^N} F(u_j)\,dx-1\right)\int_{\mathbb R^N} f(u_j)\,\phi\,dx=0\qquad\forall \phi\in C^\infty_c(\mathbb R^N)
\end{equation}
and a density argument using subcriticality shows that the above equality holds for every $\phi \in \dot H^{s}(\mathbb R^N)$.
Indeed, let  $\phi\in \dot H^s(\mathbb R^N)$ and let $(\phi_k)_{k\in\mathbb N}\subset C^\infty_c(\mathbb R^N)$ strongly converge to $\phi$ in  $\dot H^s(\mathbb R^N)$, hence strongly in $L^{2^*_s}_{loc}(\mathbb R^N)$. Since $f(u_j)\in L^{2^*_s/p}(\mathbb R^N)$ is compactly supported and $p<(N+2s)/(N-2s)=2_s^*/(2_s^*)'$, where $(2_s^*)'=2N/(N+2s)$ is the H\"older conjugate of $2_s^*$, we deduce that $f(u_j)\phi\in L^1(\mathbb R^N)$ and that
\[
\lim_{k\to+\infty}\int_{\mathbb R^N} f(u_j)\,\phi_k\,dx=\int_{\mathbb R^N} f(u_j)\,\phi\,dx.
\]
Testing \eqref{eulerlagrange} with $u_j$ yields
\[
[u_j]_{\dot H^s(\mathbb R^N)}^2+j\left(\int_{\mathbb R^N}F(u_j)\,dx-1\right)\int_{\mathbb R^N}f(u_j)\,u_j=0\qquad\forall j\in\mathbb N.
\]
Let $\theta_j:=j\left(\int_{\mathbb R^N}F(u_j)\,dx-1\right)$.  The above relation, thanks to the fact that the limit in \eqref{subcriticality3} is positive  and to the boundedness in $\dot H^s(\mathbb R^N)$ of the sequence $(u_j)_{j\in\mathbb N}$, shows that
$(\theta_j)_{j\in\mathbb N}$ is a bounded sequence of nonpositive numbers. Up to extraction of a further subsequence, it converges to some $\theta\le 0$, therefore by passing to the limit in \eqref{eulerlagrange}, thanks to \eqref{subvarphi}, we get
\[
\langle u,\phi\rangle_{\dot H^s(\mathbb R^N)}+\theta\int_{\mathbb R^N} f(u)\phi\,dx=0\qquad\forall \phi\in C^\infty_c(\mathbb R^N),
\]
which again extends by density to every test function in $\dot H^s(\mathbb R^N)$.
 This shows that $u$ satisfies \eqref{energyequation} with $a=-\theta$.
 We  check that $\theta\neq 0$. Indeed, if $\theta=0$ we obtain $(-\Delta)^s u=0$ in $\mathbb R^N$, and since $u\ge0$ we deduce $u\equiv 0$ by Liouville theorem (see \emph{e.g.} \cite[Theorem 1.1]{Fall}). But then the limit is zero in \eqref{subcriticality2}, thus \eqref{subcriticality4} holds, again contradicting the boundedness of the sequence $(\mathcal G_j(u_j))_{j\in\mathbb N}$.
A solution to \eqref{energyequation} with an arbitrary $a>0$ is given by  the rescaled function $v(x)=u(\lambda x)$ with $\lambda=(-a/\theta)^{\frac1{2s}}$. }

\bigskip


In order to obtain a solution in the sense of Definition \ref{weakenergy}, we are left to prove the boundedness of $u$. 
 Up to rescaling, we can assume w.l.o.g. that $u$ solves \eqref{energyequation} with $a=1$. By the condition $p<(N+2s)/(N-2s)$, it follows that $f(u)\in L^{q}(\R^{N})$ for $q=2N/(N+2s)$. Thus \cite[Corollary 1.4]{Fall} implies that $u=(-\Delta)^{-s}(f(u))$.
Since $u\in L^{2^{*}_{s}}$,  we have that $f(u)\in L^{q_0}$, where $q_0=2^{*}_{s}/p>1$. Then, by the same argument,  $u=(-\Delta)^{-s}(f(u))\in L^{q_1}$, where
\[
q_1:=\frac{Nq_0}{N-2sq_0}=\frac{2N}{p(N-2s)-4s}>2^{*}_{s}.
\]
Bootstrapping, after a finite number of steps we have that $f(u)\in L^{q}$, for some $q>N/2s$. Now, since $f(u)$ is supported in a ball $B_{R}:=\left\{u>\mathcal{C}\right\}$ we find
\[
u(x)=c_{N,s}\int_{\R^{N}}\frac{f(u(|\tilde x|))}{|x-\tilde x|^{N-2s}}\,d\tilde x=\int_{B_{R}}\frac{f(u(|\tilde x|))}{|x-\tilde x|^{N-2s}}\,d\tilde x,
\]
thus the radial monotonicity of $u$ and H\"older inequality yields
\[
\|u\|_{L^\infty}=u(0)=c_{N,s}\int_{B_{R}}\frac{f(u(|\tilde x|))}{|\tilde x|^{N-2s}}\,d\tilde x
\leq \|f(u)\|_{q}\left(\int_{B_{R}}\frac{d\tilde x}{|\tilde x|^{\frac{q(N-2s)}{q-1}}}\,d\tilde x\right)^{(q-1)/q},
\]
where the integral at the right hand side is finite since $q>N/2s$.
\end{proof}

\begin{remark} 
{By a maximum principle argument, we have that any weak energy solution $u$ to \eqref{mainequation} is nonnegative (see, for instance, Section 4 in \cite{RosOton:survey}). Moreover, it is possible to apply \cite[Theorem 1.2]{Felmer}, based on moving plane arguments, to show that $u$ is always radially decreasing.}
\end{remark}

\begin{remark}
An alternative proof for Proposition \ref{existence} can be given, at least for $1\le p<\tfrac{N}{N-2s}$, by establishing existence of minimizers $\rho$ for functional \eqref{functional} with $m=1+\tfrac1p$, see Remark \ref{existencebis} later on. Indeed, such minimizers satisfy the  Euler-Lagrange equation \eqref{steadyintro}, which is equivalent to \eqref{mainequation} for some constant $\mathcal{C}_{1}$ with $u=(-\Delta)^{-s}\rho$ as discussed in Section \ref{uniqsteady}. Then using the scaling property of \eqref{mainequation} allows to find a solution to the same equation for a given constant $\mathcal{C}$.
\end{remark}


Let us consider now the regularity of ground states for \eqref{mainequation}. We use  the  convention $C^\alpha=C^{\lfloor\alpha\rfloor,\alpha-\lfloor\alpha\rfloor}$ for H\"older spaces, where $\alpha>0$ and $\lfloor\alpha\rfloor:=\max\{z\in\mathbb Z:z<\alpha\}$. We recall first the interior \emph{a priori} estimates of Ros-Oton and Serra \cite{RosOton-Serra3}.  

\begin{proposition}[\cite{RosOton-Serra3}]
\label{prop:schauder}
If $u\in C^\infty(\R^N)$ solves
$$\Ds u=h\quad \text{in }B_1$$
then, for any $\beta\in(0,2s)$, there is a positive constant $\mathsf{C}$ depending on $n$, $s$ and $\beta$ such that
\begin{equation}\label{eq:RS-1}
\norm[ C^\beta(\overline{B_{1/2}})]{u}
\leq  \mathsf{C}\left(
    \norm[L^\infty(\R^N)]{u}
    +\norm[L^\infty(B_1)]{h}
\right),
\end{equation}
and
\begin{equation}\label{eq:RS-3}
\norm[ C^\beta(\overline{B_{1/2}})]{u}
\leq  \mathsf{C}\left(
    \norm[L^\infty(B_1)]{u}
    +\norm[L^\infty(B_1)]{h}
    +\norm[L^1_s(\R^N)]{u}
\right),
\end{equation}
Moreover, given $\beta>0$, if neither $\beta$ nor $\beta+2s$ is an integer, then
\begin{equation}\label{eq:RS-2}
\norm[C^{\beta+2s}(\overline{B_{1/2}})]{u}
\leq  \mathsf{C}\left(
    \norm[C^{\beta}(\R^N)]{u}
    +\norm[C^{\beta}(B_1)]{h}
\right),
\end{equation}
and
\begin{equation*}\label{eq:RS-4}
\norm[C^{\beta+2s}(\overline{B_{1/2}})]{u}
\leq  \mathsf{C}\left(
    \norm[C^{\beta}(\overline{B_1})]{u}
    +\norm[C^{\beta}(\overline{B_1})]{h}
    +\norm[L^1_s(\R^N)]{u}
\right).
\end{equation*}
\end{proposition}

\begin{proposition}\label{regularity}
Let $u$ be a ground state for \eqref{mainequation}, with $p\in[1,(N+2s)/(N-2s))$ and $\mathcal C>0$. If $p+2s\notin\mathbb N$, then  $u\in C^{p+2s}(\mathbb R^N)$. Else if $p+2s\in\mathbb N$, then
$u\in C^{p+2s-\varepsilon}(\mathbb R^N)$ for any small $\varepsilon>0$.
 Moreover,  $u$ is $C^{\infty}$ in the set $\left\{u>\mathcal{C}\right\}$.
 In particular, $\Ds u$ is pointwise well-defined as a singular integral by means of \eqref{singularintegral}.
\end{proposition}

\begin{proof}
Weak energy solutions  to \eqref{mainequation},  are shown to be Lipschitz on
$\mathbb R^N$ for $p\ge 1$ in \cite[Theorem 8]{CHMV} (i.e., for $m\le 2$, see also Theorem \ref{thm:regmin}  in Section \ref{uniqsteady}). Then, it is enough to follow the same arguments in \cite{CHMV}, making use of  the \emph{a priori} estimate \eqref{eq:RS-2}. Assuming that  $p=1$, we obtain $u\in C^{1,2s}(\mathbb R^N)$ if $2s<1$, $u\in C^{1,\alpha}(\mathbb R^N)$ for any $\alpha\in(0,1)$ if $2s=1$, and $u\in C^{2,2s-1}(\mathbb R^N)$ if $1<2s<2$.
Since the nonlinearity $u\mapsto (u-\mathcal C)_+^p$ is in $C^p_{\rm loc}(\mathbb R)$, a bootstrap argument based on \eqref{eq:RS-2} and on \eqref{mainequation}  yields the result for $p>1$.
 We refer to \cite[Theorem 10]{CHMV} for the smoothness in the interior of the support, which is obtained again by a bootstrap argument based on the same \emph{a priori} estimates. The last statement is then a consequence of \cite[Proposition 2.4]{Silvestre:regularity}.
\end{proof}

We finally give some remarks on the interpretation in terms of the fractional aggregation-diffusion equation \eqref{eq:KS}.   If $u$ is a solution to \eqref{mainequation} as given by Proposition \ref{existence}, then we clearly obtain that $\rho:=(-\Delta)^s u$ is supported on a ball, i.e., $$\{x\in\mathbb R^N:\rho(x)>0\}=\{x\in\mathbb R^N:u(x)>\mathcal{C}\}=:B_{R(u)}.$$
The value $R=R(u)$ (radius of the \emph{free boundary}) is well defined since any nontrivial radially decreasing solution is actually strictly decreasing at the value $\mathcal C$. Indeed,
 suppose by contradiction that $0\le R_1<R_2$ exist  such that $u(x)=\mathcal{C}$ in $$\mathfrak{A}=\left\{x: R_{1}<|x|<R_{2}\right\}.$$ Then, by setting
$
v:=u-\mathcal{C}
$
we have that on $\mathfrak{A}$ there hold
\[
(-\Delta)^{s}v=0\quad\mbox{and}\quad v=0,
\]
thus by the continuation property of the fractional Laplacian from  \cite{Salo} (Theorem \ref{thm:unique-continuation} below) we find $v\equiv 0$ everywhere, i.e., $u\equiv\mathcal{C}$ in $\R^{N}$, which is a contradiction.

In addition, since $\rho$ is bounded and compactly supported  and since $u$ satisfies
\[
u(x)=(-\Delta)^{-s}\rho=c_{N,s}\int_{\R^{N}}\frac{\rho(y)}{|x-y|^{N-2s}}\,dy,
\]
we find
\begin{equation}\label{asympbehM}
\lim_{|x|\rightarrow\infty}\frac{|x|^{N-2s}u(x)}{c_{N,s}}=\int_{B_{R}(u)}\rho(x)\,dx=:M
\end{equation}
In other words, we necessarily have a precise decay rate at infinity
$$u(x)\sim M\,c_{N,s}|x|^{2s-N},$$
where the constant $M$ corresponds to the  \emph{mass condition}
\[
\int_{\mathbb R^N}a(u-\mathcal C)^p_+\,dx=\int_{B_{R}(u)}a(u-\mathcal C)^p\,dx=M.
\]
The mentioned unique continuation property from \cite{Salo} is the following
\begin{theorem} \cite[Theorem 1.2]{Salo}\label{thm:unique-continuation}
For $s\in(0,1)$, if $u\in H^{r}(\mathbb R^n)$ for some $r\in\mathbb  R$, and if
both $u$ and $(-\Delta)^s u$  vanish in some open set, then $u \equiv 0$.
\end{theorem}

\subsection{The critical and supercritical regimes}

 In these cases we will find (see Proposition \ref{$C=0$}) that, in order to get  ground state for \eqref{mainequation}, we must necessarily choose $\mathcal{C}=0$. But for this choice,  there are positive solutions to \eqref{mainequation}  with the asymptotic behavior $|x|^{-2s/(p-1)}$ near infinity, 
which do not belong, due to the slow decay for $p$ large, neither to $\dot H^{s}(\mathbb R^N)$ nor to $L^{1}(\mathbb R^N)$. Thus one needs to consider  positive distributional solutions for the equation
 \begin{equation}\label{equation-cero}
(-\Delta)^s u=au^p\qquad\mbox{in $\mathbb R^N$}
\end{equation}
in the sense of Definition \ref{distrib}.
We have existence of distributional solutions, but a discussion on this topic will be postponed until Section \ref{critsupercrreg}.

Now we use the above {\emph{a priori}} estimates from Proposition \ref{prop:schauder} to get  smoothness of bounded positive distributional solutions for \eqref{equation-cero}.

\begin{proposition}
\label{prop:apriori}
If $u\in C^\infty(\R^N)$ is a bounded solution to \eqref{equation-cero}, then for any $\beta>0$ which is not an integer,
\[
\norm[C^{\beta}(\R^N)]{u}
\leq \mathsf{C},
\]
for a constant $\mathsf{C}$ depending only on $n$, $s$, $\beta$ and $\norm[L^\infty(\R^N)]{u}$.
\end{proposition}

\begin{proof}
Suppose $\beta\in(0,2s)$. Fix any center $x_0\in \R^N$. Since $u^p\in L^\infty(\R^N)$, applying \eqref{eq:RS-1} to $B_1(x_0)$, we have
\[
\norm[C^{\beta}(\overline{B_{1/2}(x_0)})]{u}
\leq \mathsf{C}\left(
    \norm[L^\infty(\R^N)]{u}
    +\norm[L^\infty(B_1(x_0))]{u^p}
\right)
\leq \mathsf{C}.
\]
Since $\mathsf{C}$ is independent of $x_0$, a standard covering argument implies
\begin{equation}\label{eq:reg-1}
\norm[C^\beta(\R^N)]{u}
\leq \mathsf{C}.
\end{equation}
If $\beta\in[2s,4s)$, we observe that, by \eqref{eq:reg-1},
\[
\norm[C^{\beta-2s}(B_1(x_0))]{u^p}
\leq \mathsf{C}\norm[L^\infty(B_1(x_0))]{u}^{p-1}
    \norm[C^{\beta-2s}(B_1(x_0))]{u}
\leq \mathsf{C}
\]
and then we use \eqref{eq:RS-2} and \eqref{eq:reg-1} to obtain
\[
\norm[C^{\beta}(\overline{B_{1/2}(x_0)})]{u}
\leq \mathsf{C}\left(
    \norm[C^{\beta-2s}(\R^N)]{u}
    +\norm[C^{\beta-2s}(B_1(x_0))]{u^p}
\right)
\leq \mathsf{C},
\]
provided that neither $\beta$ nor $\beta-2s\in(0,2s)$ is an integer (otherwise, we replace $\beta-2s$ by another non-integer number in $(0,2s)$). A covering and an inductive argument yield the desired \emph{a priori} estimate.
\end{proof}

\begin{corollary}\label{coro}
If $u\in L^\infty(\R^N)$ is a distributional solution of \eqref{equation-cero}, then $u\in C^\infty(\R^N)$.
\end{corollary}

\begin{proof}
This follows from a standard mollification argument. Given any $\phi \in C^\infty_c(\mathbb R^N)$,
let $\phi_\epsilon=\eta_\epsilon\ast\phi$, where $\eta_\epsilon(x)=\epsilon^{-N}\eta(x/\epsilon)$ and $\eta$ is the standard unit mollifier. Now set $u_\epsilon=u\ast\eta_\epsilon$ and $(u^p)_\epsilon=u^p\ast\eta_\epsilon$. We have that the convolution commutes with $\Ds$, thus
\[
\int_{\mathbb R^N} u_\epsilon\Ds\phi\,dx=\int_{\mathbb R^N}u\Ds \phi_\epsilon\,dx=\int_{\mathbb R^N} au^p\phi_\eps\,dx=\int_{\mathbb R^N}a(u^p)_\epsilon\,\phi\,dx,
\]
so that $\Ds u_\epsilon=a(u^p)_\epsilon$ holds distributionally. Since $\|(u^p)_\eps\|_{L^\infty}\le\|u^p\|_{L^\infty}$, from Proposition \ref{prop:apriori} we obtain local H\"older estimates for $u_\epsilon$ that are independent of $\epsilon$. We conclude that the limit $u$ is a smooth function.
\end{proof}

%
The following Proposition describes the weak algebraic decay of the ground states in the supercritical case.
\begin{proposition}\label{upperboundsdecay}
Let $p>(N+2s)/(N-2s)$ and let $u\in L^\infty(\mathbb R^N)$ be a ground state to \eqref{equation-cero}. Then there exists $\mathsf{C}=\mathsf{C}_{N,s,p}$ such that
 $$u(x)\le \mathsf{C}_{N,s,p}\, a^{-\tfrac{1}{p-1}}|x|^{-\tfrac{2s}{p-1}}$$
  for any $x\in\mathbb R^N$.
\end{proposition}
\begin{proof}
Recall that $u$ is positive and vanishes at infinity by hypothesis. By Proposition \ref{dec} and Corollary \ref{coro}, it is also smooth and both the equalities $\Ds u=u^p$ and $u=(-\Delta)^{-s}u^p$ hold distributionally and pointwise everywhere.
From the second equation, since $u$ is radially decreasing, we get
\[
u(x)\ge c_{N,s}\int_{B_{|x|}(0)}\frac{au(y)^p}{|x-y|^{N-2s}}\,dy\ge c_{N,s}\int_{B_{|x|}(0)}\frac{au(x)^p}{(2|x|)^{N-2s}}\,dy=\mathsf{C}_{N,s}\,a\,u(x)^p |x|^{2s},
\]
hence $u(x)\le \mathsf{C}_{N,s,p}\,a^{-\tfrac{1}{p-1}}|x|^{-\tfrac{2s}{p-1}}$, as claimed.
\end{proof}

\begin{remark}
A similar proof gives the decay of the derivatives. Also, the constant $\mathsf{C}_{N,s,p}$ is not sharp.
\end{remark}

Next we show that the Poisson extension \eqref{Poisson-extension} inherits the decay of the function, a fact that  we will employ in Section \ref{critsupercrreg}. Let us write $\angles{x}=\sqrt{1+|x|^2}$.

\begin{lemma}\label{lem:ext-est}
Suppose $0\leq u(x)\leq \angles{x}^{-\alpha}$, for some $\alpha\in[0,N)$ and let $U$ be the Poisson extension of $u$. Then there is a universal constant $\mathsf{C}= \mathsf{C}(N,s)$ such that, for all $x\in\R^N$ and $y>0$,
\[
0\leq U(x,y)\leq \mathsf{C}\angles{x}^{-\alpha}.
\]
\end{lemma}

\begin{proof}
We borrow an idea from \cite[Lemma 4.7 (1)]{CLW}. We first deal with the case $|x|\geq1$ and write $U(x,y)=I_1+I_2+I_3$, where (up to multiplicative constant)
\[\begin{split}
I_1&=
\int_{
    \set{|x-\zeta|<\frac{|x|}{2}}
    \cup\set{|x-\zeta|>2|x|}
}
    \dfrac{
        y^{2s}
    }{
        (|x-\zeta|^2+y^2)^{\frac{N+2s}{2}}
    }
    u(\zeta)
\,d\zeta,\\
I_2&=
\int_{
    \set{\frac{|x|}{2}<|x-\zeta|<2|x|}
    \cap\set{|\zeta|<1}
}
    \dfrac{
        y^{2s}
    }{
        (|x-\zeta|^2+y^2)^{\frac{N+2s}{2}}
    }
    u(\zeta)
\,d\zeta,\\
I_3&=
\int_{
    \set{\frac{|x|}{2}<|x-\zeta|<2|x|}
    \cap\set{|\zeta|> 1}
}
    \dfrac{
        y^{2s}
    }{
        (|x-\zeta|^2+y^2)^{\frac{N+2s}{2}}
    }
    u(\zeta)
\,d\zeta.
\end{split}\]
For $I_1$ we use the decay of $u(\zeta)$ to estimate
\[\begin{split}
I_1
&\leq C|x|^{-\alpha}
    \int_{\R^N}
        \dfrac{
            y^{2s}
        }{
            (|x-\zeta|^2+y^2)^{\frac{N+2s}{2}}
        }
    \,d\zeta
\leq \mathsf{C}|x|^{-\alpha}.
\end{split}\]
For $I_2$ we freeze the kernel and use the $O(1)$ bound on $u(\zeta)$ to see that
\[\begin{split}
I_2
&\leq
    \dfrac{
        y^{2s}
    }{
        (|x|^2+y^2)^{\frac{N+2s}{2}}
    }
    \int_{\{|\zeta|<1\}}
        \mathsf{C}
    \,d\zeta
\leq
    \mathsf{C}\dfrac{
        y^{2s}
    }{
        (|x|^2+y^2)^{\frac{N+2s}{2}}
    }.
\end{split}\]
In $I_3$, while we freeze the kernel again, we integrate over the whole region of $\zeta$, which is contained in $\set{1<|\zeta|<3|x|}$, so that
\[\begin{split}
I_3
&\leq
    \dfrac{
        y^{2s}
    }{
        (|x|^2+y^2)^{\frac{N+2s}{2}}
    }
    \int_{\{1<|\zeta|<3|x|\}}
        \mathsf{C}|x|^{-\alpha}
    \,d\zeta
\leq
     \mathsf{C}\dfrac{
        |x|^{N-\alpha}y^{2s}
    }{
        (|x|^2+y^2)^{\frac{N+2s}{2}}
    }.
\end{split}\]
Since $\alpha<N$, the estimate for $I_3$ dominates that for $I_2$. Moreover, by using the common upper bound $(|x|^2+y^2)^{\frac12}$ for both $|x|$ and $y$ we see that
\[
\dfrac{
    |x|^{N}y^{2s}
}{
    (|x|^2+y^2)^{\frac{N+2s}{2}}
}
\leq 1.
\]
Combining these estimates we conclude that, for $|x|\geq 1$,
\[
U(x,y)
\leq
     \mathsf{C}|x|^{-\alpha}
    + \mathsf{C}\dfrac{
        (1+|x|^{N-\alpha})y^{2s}
    }{
        (|x|^2+y^2)^{\frac{N+2s}{2}}
    }
\leq
     \mathsf{C}|x|^{-\alpha}
    \left(
    1+\dfrac{
        |x|^N y^{2s}
    }{
        (|x|^2+y^2)^{\frac{N+2s}{2}}
    }
    \right)
\leq
     \mathsf{C}|x|^{-\alpha}.
\]
For $|x|\leq1$, we simply use the Young's convolution inequality and the fact that the Poisson kernel integrates to $1$ to see
\[
U(x,y)
\leq  \mathsf{C}\norm[L^\infty(\R^N)]{u}
    \int_{\R^N}
        \dfrac{
            y^{2s}
        }{
            (|x-\zeta|^2+y^2)^{\frac{N+2s}{2}}
        }
    \,d\zeta
\leq  \mathsf{C}.
\]
In summary,
\[
0\leq U(x,y)\leq  \mathsf{C}\min\left\{|x|^{-\alpha},1\right\},
\]
and this completes the proof.
\end{proof}

As a consequence, we have the following decay estimates, which turn out to be useful in the proof of Theorem \ref{th:super}:

\begin{proposition}\label{prop:d-decay}
Suppose $u\in L^\infty(\R^N)$ is a distributional solution  to \eqref{equation-cero}  satisfying
\[
0\leq u(x)\leq \mathsf{C}\angles{x}^{-\frac{2s}{p-1}}.
\]
Then
\[
|x||Du(x)|+|x|^2|D^2u(x)|
\leq \mathsf{C}\angles{x}^{-\frac{2s}{p-1}}.
\]
\end{proposition}

\begin{proof}
Fix $|x|\geq 1$ and write $\rho=|x|/2$. The function
$u_\rho(z)=u(x+\rho z)$ satisfies, for some positive constant $\mathsf{C}$,
\[
\norm[L^\infty(B_1)]{u_\rho}
\leq \mathsf{C}\rho^{-\frac{2s}{p-1}}
\]
and
\[
\Ds u_\rho=\rho^{2s}u_\rho(z)^p
    \quad \textin \R^N.
\]
For any $\beta\in(0,2s)$, by \eqref{eq:RS-3},
\[\begin{split}
\norm[C^\beta(\overline{B_{1/2}})]{u_\rho}
&\leq \mathsf{C}\left(
    \norm[L^\infty(B_1)]{u_\rho}
    +\norm[L^\infty(B_1)]{h}
    +\norm[L^1_s(\R^N)]{u_\rho}
\right)\\
&\leq \mathsf{C}\left(
    \rho^{-\frac{2s}{p-1}}
    +\rho^{2s}\rho^{-\frac{2sp}{p-1}}
    +\rho^{-\frac{2s}{p-1}}
\right)\\
&\leq \mathsf{C}\rho^{-\frac{2s}{p-1}},
\end{split}\]
where the estimate for the $L^1_s$ norm follows from Lemma \ref{lem:ext-est} with $y=1$. An bootstrap argument as in Proposition \ref{prop:apriori} shows that the same estimate is true for any $\beta>0$, in particular for $\beta=1$ and $\beta=2$.
\end{proof}

\section{Uniqueness in the subcritical case: proof of Theorem \ref{maintheorem}}\label{uniqsubcase}
We are now in a position to prove Theorem \ref{maintheorem} in the subcritical case $1\le p< (N+2s)/(N-2s)$. Recall that in this case we always have $\mathcal{C}>0$ and the existence and regularity of ground states are described in Proposition \ref{existence} and Proposition \ref{regularity}. Through the proof, we assume w.l.o.g that $a=1$.
\\[0.5pt]

\noindent \emph{Proof of Theorem \ref{maintheorem}. }
Let $u_1,u_2$ be two ground states to equation \eqref{mainequation} with $\mathcal{C}>0$ fixed. 
Now  let $v_i=u_{i}-\mathcal{C}$, $i=1,2$. As a first case, assume that  $v_{1}(0)=v_{2}(0)$. We use an approach inspired by \cite{FLeSil}, based on the use of the monotonicity formula.
We have that the difference
\[
w=v_{1}-v_{2}
\]
satisfies the equation
\begin{equation}
(-\Delta)^{s}w=\mathcal{V}(r)w\label{eqV1},
\end{equation}
where the potential $\mathcal{V}$ is defined through
\[
\mathcal{V}(r)=\frac{g(v_1(r))-g(v_2(r))}{v_1(r)-v_2(r)},\quad \text{for}\quad g(t)=t^p_+.
\]
Let $R_{i}$ be the radius of the ball $\left\{v_{i}>0\right\}$ for $i=1,2$. {We recall that $v_1$ and $v_2$ are $C^1$ functions thanks to Proposition \ref{regularity}. }

We first prove that $R_1=R_2$. {Suppose by contradiction that $R_1\neq R_2$ and w.l.o.g. assume that $R_{1}<R_{2}$ (we will reach a contradiction by showing that $v_1\equiv v_2$)}. Notice that $\mathcal{V}$ is nonnegative and continuous for $r\geq0$, moreover $V\equiv0$ for $r\geq R_{2}$. In the linear case, \emph{i.e.} $p=1$, we clearly have
\begin{equation*}
\mathcal{V}(r)=\begin{cases}
1
\,
  & \text{if}\,r\in[0,R_1), \\
  -\frac{v_2}{v_1-v_2}\,
  &\text{if}\,r\in[R_1,R_2],\\
 0\,
  &\text{if}\,r> R_2.
 \end{cases}
\end{equation*}
Then $\mathcal{V}^{\prime}\equiv0$ for $r<R_{1}$, while for $r\in(R_1,R_2)$
\[
\mathcal{V}'(r)=\frac{v_2}{(v_1-v_2)^{2}}v^{\prime}_{1}-\frac{v_1}{(v_1-v_2)^{2}}v^{\prime}_{2}<0,
\]
then $\mathcal{V}$ is decreasing for $r\geq0$. In the case $p>1$, writing
\[
\mathcal{V}(r)=\mathsf{a}(v_{1}(r),v_{2}(r))
\]
where
\[
\mathsf{a}(x,y)=\frac{g(x)-g(y)}{x-y},
\]
we can write in the interval $[0,R_2]$
\begin{equation*}
\mathcal{V}(r)=\begin{cases}
\mathsf{a}(v_{1}(r),v_{2}(r))
\,
  & \text{if}\,v_1(r)\neq v_{2}(r), \\
  g'(\alpha)\,
  &\text{if}\,\alpha:=v_{1}(r)=v_{2}(r).
 \end{cases}
\end{equation*}
An explicit computation gives
\begin{equation}
\mathcal{V}^{\prime}(r)=\mathsf{a}_{x}(v_{1}(r),v_{2}(r))v_{1}^{\prime}(r)+\mathsf{a}_{y}(v_{1}(r),v_{2}(r))v^{\prime}_{2}(r).\label{derivV}
\end{equation}
But using Taylor's formula,
\begin{equation}\label{derivforg}
\mathsf{a}_{x}(x,y)=\begin{cases}
\frac{g^{\prime}(x)(x-y)-g(x)+g(y)}{(x-y)^{2}}
\,
  & \text{if}\,x\neq0, \\
  \frac{1}{2}g''(x)\,
  &\text{if}\,x=y\neq0.
 \end{cases}
\end{equation}
and by the convexity of $g(t)=t_{+}^{p}$ for $t>0$ we find $\mathsf{a}_{x}(v_{1}(r),v_{2}(r))\geq 0$ when $r \in (0,R_{1})$. When $r\in [R_{1},R_{2})$ we have $v_{1}(r)\leq0$ and $v_{2}(r)>0$, then $g^{\prime}(v_{1}(r))=g(v_{1}(r))=0$ whence
\[
\mathsf{a}_{x}(v_{1}(r),v_{2}(r))=\frac{g(v_{2}(r))}{(v_{1}(r)-v_{2}(r))^{2}}\geq0.
\]
Analogously we have $\mathsf{a}_{y}(v_{1}(r),v_{2}(r))\geq0$ thus by the fact that the $v_{i}$ are radially decreasing by \eqref{derivV} we find
$\mathcal{V}^{\prime}(r)\leq  0$. {Summarizing we have that $\mathcal V$ is decreasing for $r>0$ and differentiable at any $r>0$, possibly except $r=R_i$, $i=1,2$.}

Now we consider the Caffarelli-Silvestre extension problem \eqref{extension} corresponding to \eqref{eqV1}, satisfied by the $s$-harmonic extension $W=W(|x|,y)$ on the upper-half space $\R^{N+1}_{+}=\R^{N}\times(0,\infty)$:
\begin{equation}\label{extension1}
\begin{cases}
W_{rr}+\frac{N-1}{r}W_{r}+\frac{1-2s}{y}W_{y}+W_{yy}=0,
	&\text{ in }\R^{N+1}_{+},\\
W(\cdot,0)=w,\\
\displaystyle-d_{s}\lim_{y\rightarrow0}y^{1-2s}W_{y}=\mathcal{V}(r)w(r),
	&\text{ in }\partial \R^{N+1}_{+}.
\end{cases}
\end{equation}
Then it is possible associate to \eqref{extension1} the following Hamiltonian
\begin{equation}\label{Hamiltonian}
\Phi(r)=\frac{d_{s}}{2}\int_{0}^{\infty}y^{1-2s}\left\{W_{r}^{2}(r,y)-W_{y}^{2}(r,y)\right\}dy+\frac{1}{2}\mathcal{V}(r)w^{2}(r).
\end{equation}
A similar argument employed to show the asymptotic estimate \cite[Proposition B.2]{FLeSil} ensures that the above Hamiltonian is well defined.
Notice also that
\[
\mathcal{V}(r)w^{2}(r)=0
\]
for $r>R_2$. Moreover, for $p=1$ using the expression of $\mathcal{V}^{\prime}(r)$ we find that $\mathcal{V}^{\prime}(r)w^{2}(r)=0$ for $r\in[0,R_1)\cup(R_2,+\infty)$ and is (strictly) negative in $(R_1,R_2)$. In the case $p>1$ and $v_{1}(r)\neq v_{2}(r)$ by \eqref{derivV}-\eqref{derivforg} we find
\[
\mathcal{V}^{\prime}(r)w^{2}(r)=\frac{d}{dr}\left[(g(v_{1}(r))-g(v_{2}(r)))(v_{1}(r)-v_{2}(r))\right]-2[(g(v_{1}(r))-g(v_{2}(r)))(v^{\prime}_{1}(r)-v_{2}^{\prime}(r))]
\]
and notice that
\[
\lim_{r\rightarrow R_2}\frac{(g(v_{1}(r))-g(v_{2}(r)))(v_{1}(r)-v_{2}(r))}{r-R_2}=0,
\]
thus the function $\mathcal{V}^{\prime}(r)w^{2}(r)$ can be extended continuously to 0 at $r=R_{2}$.
\\[0.5pt]
Now, using \eqref{extension}, we can compute the derivative of $\Phi$ along the flow, which is given in \cite{FLeSil} but we give here some details for the sake of completeness. We have
\begin{align*}
\frac{d\Phi}{dr}=d_{s}\int_{0}^{\infty}y^{1-2s}\left[W_r\,W_{rr}-W_{y}W_{ry}\right]dy+\mathcal{V}w\,w^{\prime}+\frac{1}{2}\mathcal{V}^{\prime}w^{2},
\end{align*}
hence using the extension equation in \eqref{extension}
\begin{align*}
&\frac{d\Phi}{dr}=-d_{s}\int_{0}^{\infty}y^{1-2s}W_{r}^{2}\,dy-d_{s}\int_{0}^{\infty}\frac{d}{dy}\left(y^{1-2s}W_{r}W_{y}\right)dy+
\mathcal{V}w\,w^{\prime}+\frac{1}{2}\mathcal{V}^{\prime}w^{2}.
\end{align*}
Then the boundary condition in \eqref{extension} implies
\begin{align*}
&\frac{d\Phi}{dr}=-d_{s}\,\frac{N-1}{r}\int_{0}^{\infty}y^{1-2s}W_{r}^{2}(r,y)\,dy+\varphi(r),
\end{align*}
where $\varphi(r)$ is the \emph{a.e.} continuous function defined through
\begin{equation*}
\varphi(r)=\begin{cases}
\frac{1}{2}\mathcal{V}^{\prime}(r)w^{2}(r),
	&\text{ for }r\leq R_{2},\\
\\
0,
	&\text{ for }r\geq R_{2}.
\end{cases}
\end{equation*}
Therefore, since $\mathcal{V}(r)$ is radially decreasing and $\Phi$ is continuous, we have that $\Phi$ is decreasing for $r\geq0$.

We next claim that, for the Hamiltonian \eqref{Hamiltonian},
\begin{equation}\label{decay-Hamiltonian}
\lim_{r\rightarrow\infty}\Phi(r)=0.
\end{equation}
Indeed, observe that  clearly
\[\begin{split}
\int_{\R^N}\int_0^\infty
    y^{1-2s}(|\nabla_x W|^2+W_y^2)
\,dx\,dy
&=\mathsf{C}(N,s)\int_{\R^N}|(-\Delta)^{s/2} w|^{2} \,dx<\infty.
\end{split}\]
{ But then we have}
\[
\int_{1}^{\infty}|\Phi(r)|\,dr\leq \frac{d_{s}}{2} \int_{1}^{\infty} \int_0^\infty
    r^{N-1}\,y^{1-2s}(|W_{r}|^2+W_y^2)
\,dr\,dy
+\frac{1}{2}\int_{1}^{R_2} \mathcal{V}(r)w^{2}(r)\,dr<\infty,
\]
{therefore $\Phi$ is in $L^{1}(1,\infty)$. Hence, keeping in mind that $\Phi$ is decreasing}, \eqref{decay-Hamiltonian} follows.
Then, as we have
\[
\Phi(0)=-\frac{d_{s}}{2}\int_{0}^{\infty}y^{1-2s}W_{y}^{2}(0,y)\,dy+\frac{1}{2}\mathcal{V}(0)w(0)^{2}\leq\frac{1}{2}\mathcal{V}(0)w(0)^{2},
\]
we find
\[
\frac{1}{2}\mathcal{V}(0)w(0)^{2}\geq \Phi(0)\geq\Phi(r)\geq \lim_{r\rightarrow\infty}\Phi(r)=0.
\]
Now since $w(0)=0$, the previous inequality gives $\Phi(r)\equiv0$ and consequently,
\[
\frac{d\Phi}{dr}=0,
\]
namely
\[
d_{s}\,\frac{N-1}{r}\int_{0}^{\infty}y^{1-2s}W_{r}^{2}(r,y)\,dy=\varphi(r)\leq0,
\]
which implies
{
\[
W_{r}\equiv0.
\]
But then we have that $w(r)$ is constant, thus $w\equiv w(0) = 0$
and we obtain
\[
v_{1}\equiv v_{2},
\]
a contradiction to the assumption $R_1<R_2$. }

{We have shown that  $R_1=R_2=:R$, thus $\mathcal{V}^{\prime}$ is singular only on the boundary of the common positivity set of $v_{1}$, $v_{2}$, i.e., at $r=R$. In any case, this does not prevent to repeat the previous argument and to conclude again $v_1\equiv v_2$, and this ends the proof in case $v_1(0)=v_2(0).$ }

In order to complete the proof of the theorem, now assume that $v_{1}(0)\neq v_{2}(0)$, set
\[
\lambda=\frac{v_{2}(0)}{v_{1}(0)}
\]
and define the rescaled function
\[
(v_{1})_{\lambda}(x)=\lambda v_{1}(\lambda^ \frac{p-1}{2s} x).
\]
Then $(v_{1})_{\lambda}$, $v_{2}$ satisfy the same equation
$$(-\Delta)^{s} v=v_{+}^{p}$$
and $(v_{1})_{\lambda}(0)=v_{2}(0)$.
Now define the function
\[
w_{\lambda}=(v_{1})_{\lambda}-v_{2}.
\]
Clearly $w_{\lambda}(0)=0$. Taking into account that $v_i$ is a translation of $u_{i}$, $i=1,2$, we have
\[
w_{\lambda}\rightarrow  {\mathcal C(1-\lambda)}, \quad \text{as}\,\,|x|\rightarrow\infty.
\]
In any case, the previous argument can be applied to this case even if $w_\lambda$ does not tend to zero as $r\to \infty$. Indeed, the extension associated to $w_\lambda$ ({in terms of the extensions $U_1,U_2$ of $u_1$ and $u_2$, respectively}) is
\[
W_{\lambda}(x,y)=\lambda U_{1}(\lambda^{\frac{p-1}{2s}} x,\lambda^{\frac{p-1}{2s}} y)-U_{2}(x,y)+{\mathcal C(1-\lambda)}.
\]
The main point is that \eqref{decay-Hamiltonian} still holds. Thus we conclude that  $\partial_{r} W_{\lambda}\equiv0$. But the condition $w_{\lambda}(0)=0$  forces $\lambda=1$, a contradiction.
\qed

%


\section{Critical and supercritical regimes: Proof of Theorem \ref{th:super}}\label{critsupercrreg}

In this section, we assume w.l.o.g. that $a=1$.
If $p\ge \tfrac{N+2s}{N-2s}$,
ground states to  \eqref{mainequation} are exactly the positive radially decreasing solutions to
\begin{equation}\label{C=0problem}
\begin{cases}
\Ds{u}=u^p &\textin \R^{N},\\
u\rightarrow 0 & |x|\rightarrow \infty.
\end{cases}
\end{equation}
In other words, for the existence of solutions we must have $\mathcal{C}=0$, as we shall prove in Proposition \ref{$C=0$} at the end of this section. The solutions to such problem are completely classified in the case $p=(N+2s)/(N-2s)$, as it is stated in cite \cite[Theorem 1.2]{CLO}:
\begin{proposition}\label{criticalcase}\cite[Theorem 1.2]{CLO}.
When $p=(N+2s)/(N-2s)$, any positive solution to \eqref{C=0problem}
is radially symmetric (up to translations) and given by
\[
u(x)=c(N,s)(t/(t^2+|x|^2))^{\frac{N+2s}{2}},
\]
for some universal constant $c$ and positive constants $t>0$.
\end{proposition}
The classification in \cite{CLO} is based on the moving plane method adapted to the equivalent integral equation, applied to Kelvin-type transformations of the solution.

Let us move to the supercritical regime. The existence of ground states for \eqref{C=0problem} is established in \cite{acgw}. Let us recall the result therein.

\begin{proposition}[\cite{acgw}]
\label{prop:C=0}
When $p>(N+2s)/(N-2s)$, there exist ground states to \eqref{C=0problem} with precise (slow) decay
\begin{equation}\label{eq:supercrit-decay}
u(x)\sim c(N,s,p)|x|^{-\frac{2s}{p-1}}
    \quad \text{ as } |x|\to\infty,
\end{equation}
where the constant $c(N,s,p)$ is given by
\[
c(N,s,p) = \left[ \frac{\Gamma(\frac{N}{2}-\frac{s}{p-1})\Gamma(\frac{sp}{p-1})}
{\Gamma(\frac{s}{p-1})\Gamma(\frac{N}{2}-\frac{sp}{p-1})}2^{-2s}\right]^{1/(p-1)}.
\]
\end{proposition}

\begin{remark}\rm
The idea of the proof of Proposition \ref{prop:C=0} is the following. We refer to
 \cite{acgw} (see also \cite{Ao-Chan-DelaTorre-Fontelos-Gonzalez-Wei} for more details).
One searches for an unbounded, continuous branch of solutions of the auxiliary equation
\[\begin{cases}
\Ds u=\lambda(1+u)^p
    & \textin B_1,\\
u=0
    & \textin \R^n\setminus B_1,
\end{cases}\]
and then perform a blow-up argument along such branch. To show the decay in \eqref{eq:supercrit-decay}, we first observe that the upper bound is given by Proposition \ref{upperboundsdecay}.
Next, the Emden-Fowler transformation $v(t)=e^{-\frac{2s}{p-1}t}u(e^{-t})$, $t=-\log|x|$, solves an equation of the form
\[
P.V.\int_{\R}\mathcal{K}(t-t')[v(t)-v(t')]\,dt'
+c(N,s,p)^{p-1}v(t)=v(t)^p
    \quad \textin \R,
\]
and therefore the exact coefficient is determined as $c(N,s,p)$ via a Hamiltonian type argument.
\end{remark}

Through a scaling argument, there are infinitely many solutions of \eqref{C=0problem} in the supercritical regime.
No uniqueness result in the sense of Theorem \ref{maintheorem} can be expected. Nonetheless, one can prove the uniqueness once the maximum value (at the origin) is fixed, which is what we do next by providing the proof of Theorem \ref{th:super}. As a consequence,
all the bounded radially decreasing solutions of Eq.~\eqref{C=0problem} can be rescaled to each other and belong to the family
$\{ \lambda^{2s/(p-1)} u(\lambda x)\}_{\lambda>0}$, where $u$ is any solution.
Interestingly, if
  $x \neq 0$, the limit
\[
\lim_{\lambda \to \infty} \lambda^{2s/(p-1)} u(\lambda x) =c(N,s,p)|x|^{-\frac{2s}{p-1}}
\]
turns out to be a singular solution.


\begin{proof}[Proof of Theorem \ref{th:super}]
 Let $u_1$ and $u_2$ be ground states to $\Ds u=u^p$, with $u_1(0)=u_2(0)=b$. The uniqueness is straightforward in the case $p=(N+2s)/(N-2s)$ due to Proposition \ref{criticalcase}, moreover in this case the constraint of the central density is equivalent to the mass constraint as the solutions are in $L^{1}(\mathbb R^N)\cap \dot{H}^{s}(\R^{N})$. Thus assume that $p>(N+2s)/(N-2s)$.
We argue as in the proof of Theorem \ref{maintheorem}, but we need to be more careful about the decay at infinity. Let $w=u_1-u_2$, so that $w(0)=0$ and $w$ solves
\[
\Ds w=\mathcal{V}(r)w,
\]
where
\[
\mathcal{V}(r)
=\dfrac{
    u_1^p(r)-u_2^p(r)
}{
    u_1(r)-u_2(r)
}
=p\int_0^1\big(
    \theta u_1(r)+(1-\theta)u_2(r)
\big)^{p-1}\,d\theta
\geq 0,
\]
with
\[
\mathcal{V}'(r)
=p(p-1)
\int_0^1
\big(
    \theta u_1(r)+(1-\theta)u_2(r)
\big)^{p-2}
\big(
    \theta u_1'(r)+(1-\theta)u_2'(r)
\big)
\,d\theta
\leq 0.
\]
Consider the Poisson extension  $W$ of $w$ from \eqref{Poisson-extension}, given up to a constant multiple by
\[
W(x,y)
=\int_{\R^n}
    \dfrac{
        y^{2s}
    }{
        (|\zeta|^2+y^2)^{\frac{N+2s}{2}}
    }
    w(x-\zeta)
\,d\zeta.
\]
Since, by the regularity and decay of $u_1$ and $u_2$ from Propositions \ref{upperboundsdecay}-\ref{prop:d-decay}
\[
w(x)\leq \mathsf{C}\angles{x}^{-\frac{2s}{p-1}},
    \quad
|Dw(x)|\leq \mathsf{C}\angles{x}^{-\frac{2s}{p-1}-1},
    \quad
|D^2w(x)|\leq \mathsf{C}\angles{x}^{-\frac{2s}{p-1}-2},
\]
we have, by applying Lemma \ref{lem:ext-est} up to the second derivative,
\begin{equation}\label{eq:decay-W-1}
W(x,y)\leq \mathsf{C}\angles{x}^{-\frac{2s}{p-1}},
    \quad
|D_{x}W(x,y)|\leq \mathsf{C}\angles{x}^{-\frac{2s}{p-1}-1},
    \quad
|D_{x}^2W(x,y)|\leq \mathsf{C}\angles{x}^{-\frac{2s}{p-1}-2},
\end{equation}
for a constant $\mathsf{C}$ independent of $y$. On the other hand, from the expression
\[
W(x,y)
=\int_{\R^n}
    \dfrac{
        y^{2s}
    }{
        (|x-\zeta|^2+y^2)^{\frac{N+2s}{2}}
    }
    w(\zeta)
\,d\zeta,
\]
derivatives in $x$ or $y$ hit the kernel and produce a decay in $y$, namely
\begin{equation}\label{eq:decay-W-2}
y\big(
    |D_xW(x,y)|+W_y(x,y)|
\big)
\leq CW(x,y)
\leq Cw(x)
\leq C\angles{x}^{-\frac{2s}{p-1}}.
\end{equation}
In addition, one can get similar estimates for $y^{1-2s} W_y$ by considering the conjugate equation as in Proposition  3.6 of \cite{Cabre-Sire}.\\

Define as in the proof of Theorem \ref{maintheorem} the Hamiltonian
\[
\Phi(r)
=\frac{d_{s}}{2}
    \int_{0}^{\infty}
        y^{1-2s}\left\{
        W_{r}^{2}(r,y)-W_{y}^{2}(r,y)
        \right\}
    dy
    +\frac{1}{2}\mathcal{V}(r)w^{2}(r).
\]
From estimates \eqref{eq:decay-W-2} and the bound of $y^{1-2s} W_y$ in the previous discussion, $\Phi$ is well defined and differentiable. Clearly, $\Phi(0)\leq 0$ and $\Phi'(r)\leq 0$. We will show that $\Phi$ is globally bounded. Once we have  that $\Phi(+\infty)=0$, we can proceed as in the proof of \ref{maintheorem} to conclude that $W\equiv 0$.\\
\noindent In order to show the decay as $r\to\infty$, let us split $$\Phi(r)=\Phi_1(r)+\Phi_2(r)+\frac12 \mathcal{V}(r)w^2(r),$$ where
\[
\Phi_1(r)
=\frac{d_{s}}{2}
    \int_{0}^{1}
        y^{1-2s}\left\{
        W_{r}^{2}(r,y)-W_{y}^{2}(r,y)
        \right\}
    dy,
\]
\[
\Phi_2(r)
=\frac{d_{s}}{2}
    \int_{1}^{\infty}
        y^{1-2s}\left\{
        W_{r}^{2}(r,y)-W_{y}^{2}(r,y)
        \right\}
    dy.
\]
By \eqref{eq:decay-W-2},
\[
|\Phi_2(r)|
\leq \mathsf{C}\int_{1}^{\infty}
    y^{-1-2s}W^2(r,y)
\,dy
\leq \mathsf{C}r^{-\frac{2s}{p-1}},
\]
for some constant $\mathsf{C}$.
{Testing the extension equation for $W$ against $W$ and integrating by parts, we have}
\[
\Phi_1(r)
=\dfrac{d_{s}}{2}
    \int_{0}^{1}
        y^{1-2s}\big(
            W_r^2(r,y)+W(r,y)\Delta_{x}W(r,y)
        \big)
    \,dy
    -\frac{1}{d_s}w(r)\Ds w(r)
    -W(r,1)W_y(r,1),
\]
which tends to zero in view of \eqref{eq:decay-W-1}. Hence, the proof is complete up to repeating the argument in the proof of Theorem \ref{maintheorem}.
\end{proof}
\begin{remark}
Using Theorem \ref{th:super} it is easy to show that for any ground state $u$ to $(-\Delta)^{s}u=u^{p}$ we have
\begin{equation}\label{asymptotsuper}
u\sim c(N,s,p)|x|^{-2s/(p-1)}\quad \textit{as}\,\, |x|\rightarrow\infty,
\end{equation} being $c(N,s,p)$ the constant for the solution in Proposition \ref{prop:C=0}. Indeed, let $u$ be any ground state and $u_{1}$ the ground state in Proposition \ref{prop:C=0}. Then the rescaled function
\[
u_{\lambda}(x):=\lambda u_{1}(\lambda^{\frac{p-1}{2s}}x)
\]
with the scaling factor defined as
\[
\lambda=\frac{u(0)}{u_{1}(0)}
\]
is still a ground state to the same equation and $u_{\lambda}(0)=u(0)$. Then Theorem \ref{th:super} gives $u=u_{\lambda}$, hence Theorem \ref{th:super} implies \eqref{asymptotsuper}. {In particular, if $u$ is a bounded radially decreasing  distributional solution to  $(-\Delta)^{s}u=u^{p}$ such that $u(x)=o(|x|^{-2s/(p-1)})$ as $|x|\to+\infty$, then $u\equiv 0$. }
\end{remark}
We close this section by showing that  we must necessarily have $\mathcal{C}=0$ in the critical and supercritical regimes. This is a consequence of the Pohozaev identity.

\begin{proposition}\label{$C=00$}
Assume that $p> (N+2s)/(N-2s)$, then there is no  ground state $u$ in $\dot H^{s}(\R^{N})$ to the equation \eqref{mainequation}.
\end{proposition}
\begin{proof} Suppose first  that $\mathcal C>0$ and set, as always, $B_{R}=\left\{u>\mathcal{C}\right\}$. Assume $u\in\dot H^{s}(\R^{N})$ is a ground state.  Then the following Pohozaev identity (see \cite{RosOton-Serra1}, \cite[Theorem 1.1]{DAVILA} or \cite[Proposition 4.1]{Secchi}) is valid in the whole space $\R^{N}$
\[
\frac{N-2s}{2}\int_{\R^{N}}|\xi|^{2s}|\hat{u}(\xi)|^{2}\,d\xi=\frac{N}{p+1}\int_{B_{R}}(u-\mathcal{C})_{+}^{p+1}\,dx.
\]
On the other hand, multiplying \eqref{mainequation} by $u$ and integrating by parts
\[
\int_{\R^{N}}|\xi|^{2s}|\hat{u}(\xi)|^{2}\,d\xi=\int_{B_{R}}u(u-\mathcal{C})_{+}^{p}\,dx> \int_{B_{R}}(u-\mathcal{C})_{+}^{p+1}\,dx,
\]
where the last inequality is strict because $u$ is continuous  and $u(0)>\mathcal{C}$.
Then
\[
\frac{N-2s}{2} \int_{B_{R}}(u-\mathcal{C})_{+}^{p+1}\,dx< \frac{N}{p+1} \int_{B_{R}}(u-\mathcal{C})_{+}^{p+1}\,dx,
\]
i.e.,
\[
\left(\frac{N-2s}{2}-\frac{N}{p+1}\right) \int_{B_{R}}(u-\mathcal{C})_{+}^{p+1}\,dx<0,
\]
which contradicts the condition $p> (N+2s)/(N-2s)$. Then $\mathcal{C}=0$ and $u$ solves $(-\Delta)^{s}u=u^{p}$ distributionally and $u(0)=\|u\|_{L^\infty(\mathbb R^N)}$, so $u\sim |x|^{-2s/(p-1)}$ for $|x|\rightarrow\infty$ by \eqref{asymptotsuper}, which implies in particular  $u\not\in L^{2^{\ast}_{s}}(\mathbb{ R}^N)$, contradicting $u\in \dot H^{s}(\mathbb{R}^N)$. \end{proof}
\begin{remark}\label{remarkcrit}
Inspecting the proof of Proposition \ref{$C=00$} it follows that in the critical case $p=(N+2s)/(N-2s)$ the ground states in the energy space $\dot H^{s}(\R^{N})$ are \emph{all} the nontrivial positive vanishing solutions to $(-\Delta)^{s}u=u^{p}$, classified entirely by Proposition \ref{criticalcase}.
\end{remark}
The following Proposition extends Remark \ref{remarkcrit} to the supercritical case.
\begin{proposition}\label{$C=0$}
Assume that $p>(N+2s)/(N-2s)$ and let $u$ be a ground state to equation \eqref{mainequation}. Then $\mathcal{C}=0$.
\end{proposition}
\begin{proof}
Suppose by contradiction that $\mathcal{C}>0$. Then $u$ is distributional solution to $\Ds u=(u-\mathcal{C})_+^p$ and the right hand side is in $L^1(\mathbb R^N)\cap L^\infty(\mathbb R^N)$ since $u$ is a ground state. In particular, by fractional Sobolev embedding we get $(u-\mathcal C)_+^p\in L^{\tfrac{2N}{N+2s}}(\mathbb R^N)\subset \dot H^{-s}(\mathbb R^N)$,
so that $u\in \dot H^{s}(\mathbb R^N)$. This is a  contradiction with Proposition \ref{$C=00$}.
\end{proof}
\begin{remark}{\rm  It is interesting to observe that there are also nonradial solutions in the supercritical regime, see \cite{ZLO} for a simple axially symmetric construction. This implies that Theorem \ref{th:super} does not hold in the more general class of bounded solutions vanishing at infinity.}
\end{remark}

\section{Uniqueness of steady states of aggregation-diffusion equations}\label{uniqsteady}

As an application of Theorem \ref{maintheorem}, in this section we deduce uniqueness of the steady states to the evolution equation \eqref{eq:KS}.

Before going through the full analysis of steady states (which is carried over in Subsection \ref{5.2}), let us briefly focus on the minimization of the natural free energy functional \eqref{functional} associated to \eqref{eq:KS}, in the diffusion dominated regime. In this regime
 the diffusion dominates over the aggregation in the dynamics given by \eqref{eq:KS}. By a scaling argument, this phenomenon is shown to occur only if
$$m>2-\frac{2s}{N}=:m_c$$
%

\subsection{Minimizers}
Since \eqref{eq:KS} conserves mass, it is positivity preserving and invariant by translations, we work with solutions
$\rho$ that for any time $t$ belong to the set
\begin{equation*}
\mY_{M}:=\left\{ \rho \in L_+^1(\R^N) \cap L^m(\R^N)\,,\,||\rho||_{L^1(\mathbb R^N)}=M\, ,\, \int_{\R^N} x\rho(x)\,dx=0\right\}\, .
\end{equation*}
In the diffusion-dominated regime, the minimization problem $\min_{\mathcal Y_M} \mathcal F$ has been investigated in \cite{CHMV}. The main results therein are summarized in the following
\begin{lemma}\label{lemma:basic}
 Let $m>m_c$ and $M>0$.
The functional $\mathcal{F}$ admits a minimizer over $\mathcal{Y}_M$. If $\rho\in\mathrm{argmin}_{\mathcal{Y}_M}\mathcal{F}$, then $\rho$ is continuous and bounded on $\mathbb{R}^N$, radially decreasing, compactly supported, smooth in the interior if its support, and it satisfies
\begin{equation}\label{Euler}
\rho=\left(\tfrac{m-1}{m}\right)^{\frac1{m-1}}\left(\chi\, c_{N,s}|\cdot|^{2s-N}\ast\rho\,-\,\mathcal{K}\right)^{\frac1{m-1}}_+\qquad\mbox{ in $\mathbb{R}^N$},
\end{equation}
where
\begin{equation}\label{explicitconstant}
0<\mathcal{K}:=-\frac2M\mathcal{F}[\rho]-\frac1M\frac{m-2}{m-1}\int_{\mathbb{R}^N}\rho^m(x)\,dx.
\end{equation}
Moreover, there holds
\begin{equation}\label{scaleidentity}
\frac{\chi}{2} \int_{\mathbb{R}^N}\int_{\mathbb{R}^N}c_{N,s}|x-y|^{2s-N}\rho(x)\rho(y)\,dx\,dy=\frac N{N-2s}\int_{\mathbb{R}^N}\rho^m(x)\,dx.
\end{equation}
\end{lemma}

We refer to \cite{CHMV} for the proof of the properties of Lemma \ref{lemma:basic}. In particular, \eqref{scaleidentity} follows by taking dilations $\rho_\lambda(x):=\lambda^N\rho(\lambda x)$ and optimizing with respect to $\lambda>0$, hence finding a unique optimal value $\lambda_*$, and then imposing $\lambda_*=1$ since $\rho$ is a minimizer.

Note that if $\rho$ is a minimizer from \eqref{functional} and \eqref{scaleidentity} we deduce \begin{equation}\label{fixednorm}\displaystyle\int_{\mathbb{R}^N}\rho^m\,dx=-\frac{(N-2s)(m-1)}{N(m-2)+2s}\,\mathcal{F}(\rho)\end{equation}
along with
\[
\begin{aligned}
  &\int_{\mathbb R^N} \rho^m \,dx= \frac{(m-1)(N-2s)}{(m-2)N+2ms} M\mathcal{K},\\
&\chi \int_{\mathbb{R}^N}\int_{\mathbb{R}^N}c_{N,s}|x-y|^{2s-N}\rho(x)\rho(y)\,dx\,dy
= \frac{2(m-1)N}{(m-2)N+2ms}M\mathcal{K}
\end{aligned}\]
and
\begin{equation}\label{C_M}
  \mathcal{F}(\rho) = - \frac{(m-2)N+2s}{(m-2)N+2ms} M \mathcal{K}.
\end{equation}
In fact, by combining \eqref{explicitconstant} and \eqref{fixednorm}, we deduce that the constant $\mathcal K$ is uniquely determined and depends only on the minimal value of $\mathcal{F}$
on $\mathcal Y_M$.

By letting $u:=(-\Delta)^{-s}\rho=c_{N,s}|\cdot|^{2s-N}\ast\rho $, we see that \eqref{Euler} rewrites in terms of $u$ as
 as \eqref{mainequation} after having suitably chosen  the  parameters $a$ and $\mathcal C$ therein. Before  applying  the general uniqueness theory from Theorem \ref{maintheorem}, we show how to obtain uniqueness of the minimizer of $\mathcal F$ by a direct argument, at least in case $m=2$.
Indeed, we have the following

\begin{lemma}  Let $m=2$ and $M>0$. Then there exists a unique minimizer of $\mathcal F$ over $\mathcal Y_M$.
\end{lemma}

\begin{proof} Existence is shown in Lemma \ref{lemma:basic} along with Euler-Lagrange equation and other properties of minimizers. Therefore, we are reduced to prove uniqueness.
Through the proof,
we use the notation $W(x):=\chi\,c_{N,s}|x|^{2s-N}$. By assuming   $m=2$, and by using the notation $\mathcal{F}_M$ for the minimal value of $\mathcal{F}$ over $\mathcal Y_M$, from \eqref{explicitconstant}  we see that $\mathcal{K}=-2\mathcal{F}_M/M$.

By Lemma \ref{lemma:basic}, any minimizer is radially decreasing, continuous and compactly supported.
Suppose by contradiction that there are two minimizers $\rho_1, \rho_2$ that do not coincide. Without loss of generality, assume that $\mathrm{supp}(\rho_1)\subseteq\mathrm{supp}(\rho_2)$. Since $\mathrm{supp}(\rho_1)\subseteq\mathrm{supp}(\rho_2)$, from \eqref{Euler} (taking $\mathcal{K}=-2\mathcal{F}_M/M$ into account) we have
\begin{equation}\label{doubling}
2\int_{\mathbb{R}^N}\rho_1\rho_2-\int_{\mathbb{R}^N} (W\ast\rho_2)\rho_1=\int_{\mathbb{R}^N}\rho_1\left( 2\rho_2- W\ast\rho_2\right)=\int_{\mathbb{R}^N}\rho_1\,\frac{2\mathcal{F}_M}{M}=2\mathcal{F}_M.
\end{equation}
On the other hand, let $\rho_{1/2}:=\frac12 \rho_1+\frac12\rho_2$. Then $\rho_{1/2}\in\mathcal Y_M$. By using the minimality of $\rho_1$, $\rho_2$ and \eqref{doubling} we get
\[\begin{aligned}
\mathcal{F}[\rho_{1/2}]&=\frac14\int_{\mathbb{R}^N}\rho_1^2+\frac14\int_{\mathbb{R}^N}\rho_2^2+\frac12\int_{\mathbb{R}^N}\rho_1\rho_2-\frac18\int_{\mathbb{R}^N}(W\ast\rho_1)\rho_1\\&\qquad-\frac18\int_{\mathbb{R}^N}(W\ast\rho_2)\rho_2-\frac14\int_{\mathbb{R}^N}(W\ast\rho_2)\rho_1\\& =\frac14\mathcal{F}[\rho_1]+\frac14\mathcal{F}[\rho_2] +\frac12\int_{\mathbb{R}^N}\rho_1\rho_2 -\frac14\int_{\mathbb{R}^N}(W\ast\rho_2)\rho_1
\\&=\frac14\mathcal{F}[\rho_1]+\frac14\mathcal{F}[\rho_2]+\frac14(2\mathcal{F}_M)=\mathcal{F}_M,
\end{aligned}\]
hence $\rho_{1/2}$ is itself a minimizer. From \eqref{fixednorm} we deduce
\[\int_{\mathbb{R}^N}\rho_{1/2}^2=\int_{\mathbb{R}^N}\rho_1^2=\int_{\mathbb{R}^N}\rho_2^2=-\frac{N-2s}{2s}\mathcal{F}_M.
\]
But this is a contradiction, since the Young inequality $\rho_{1/2}^2\le\frac12\rho_1^2+\frac12\rho_2^2$ is strict on a set of positive measure, as we are assuming that $\rho_1$ and $\rho_2$ are not coinciding.
\end{proof}

\subsection{Radial steady states}\label{5.2}
We shall characterize the uniqueness for \emph{radial} densities $\rho$ which are steady state of equation \eqref{eq:KS}   according to the following

\begin{definition}\label{def:steady}
We say that a  nonnegative  function $\rho\in L^\infty(\mathbb R^N)$ is a {\it radial steady state} for the
 evolution equation \eqref{eq:KS} if $\rho$ is radially decreasing and there exists $\mathcal{K}\ge 0$ such that
 \begin{equation}\label{eq:steady}
  \rho(x)^{m-1} = \frac{m-1}{m} \left( \chi(-\Delta)^{-s}\rho(x)-\mathcal{K}\right)_+\, , \qquad
 \mbox{for a.e.} \, x \in \R^N.
\end{equation}
\end{definition}

Let us preliminarily show that there is a one-to-one correspondence between radial steady states and ground states to \eqref{mainequation}, once $p$ and $m$ are related by $p=\tfrac{1}{m-1}$, see Figure \ref{f1}. Our uniqueness results in this subsection cover the range $m\in(1,2]$, since $p\ge 1$ in  our main theorems. In fact, given the form of \eqref{eq:steady}, it will be more convenient to rewrite \eqref{mainequation} as
\begin{equation}
(-\Delta)^{s}u=a_p(\chi u-\mathcal{K})^{p}_{+}\quad\text{ in }\mathbb R^N\label{groundstatem},
\end{equation}
where $p=\tfrac1{m-1}$ and $a_p:=(p+1)^{-p}$. Note that  \eqref{groundstatem}
is equivalent  to \eqref{mainequation} with  $a=a_p\chi^p$ and $\mathcal C=\mathcal K/\chi$.

\begin{proposition}\label{onetoone} Let $m\in(1,2]$.
Let $\rho$ be a radial steady state according to {\rm Definition \ref{def:steady}}. Let $u=(-\Delta)^{-s}\rho$ be the Riesz potential of $\rho$. Then, $u\in L^1_s({\mathbb R^N})\cap L^\infty(\mathbb R^N)$ and $u$ is a ground state to \eqref{groundstatem}, where $p=\tfrac1{m-1}$ and $a_p=(p+1)^{-p}$.
\end{proposition}
\begin{proof}
We preliminary notice that since $\rho\in L^\infty(\mathbb R^d)$, then  $\int_{\mathbb R^N}\rho(x)|x|^{2s-N}\,dx<+\infty$ follows from \eqref{eq:steady} if $\rho$ is nontrivial. Indeed, the Riesz
potential of a radially decreasing function is radially decreasing and therefore $\mathrm{ess}\sup (-\Delta)^{-s}\rho(x)=c_{N,s}\int_{\mathbb R^N}\rho(x)|x|^{2s-N}\,dx$. In particular, $\rho$ vanishes at infinity. Moreover, $u\in L^\infty(\mathbb R^N)$.

Suppose first that $\mathcal{K}>0$. Then Definition \ref{def:steady} implies that $\rho$ is compactly supported therefore  $u\in L^1_s(\mathbb R^N)$. By Sobolev embedding, since $\rho\in L^1(\mathbb R^N)\cap L^\infty(\mathbb R^N)$, we get $\rho\in \dot H^{-s}(\mathbb R^N)$ and $u\in \dot H^s(\mathbb R^N)$. Moreover, we multiply \eqref{eq:steady} by $\phi\in C^\infty_c(\mathbb R^N)$ and we integrate over $\mathbb R^N$;  by Plancherel theorem and  reasoning similarly  to Proposition \ref{rosa}  we get
\[\begin{aligned}
\int_{\mathbb R^N} a_p(\chi u-\mathcal{K})_+^p\,\phi&=\int_{\mathbb R^N}\rho\phi=
\int_{\mathbb R^N}(-\Delta)^{-s/2}\rho(-\Delta)^{s/2}\phi=
\int_{\mathbb R^N}(-\Delta)^{s/2}(-\Delta)^{-s}\rho(-\Delta)^{s/2}\phi\\&=
\int_{\mathbb R^N}(-\Delta)^{s/2}u(-\Delta)^{s/2}\phi=
\langle u,\phi\rangle_{\dot H^s(\mathbb R^N)}.
\end{aligned}\]
This shows that $u$ is a weak energy solution (and a distributional solution) to \eqref{groundstatem}. Since $\rho $ is radially decreasing and compactly supported, $u$  is radially decreasing and vanishing at infinity so that it is a ground state.  

Suppose instead that $\mathcal{K}=0$. Then there holds $\rho^{1/p}=a_p^{1/p}\chi(-\Delta)^{-s}\rho$ a.e. in $\mathbb R^N$, clearly implying $u\in L^1_s(\mathbb R^N)$. Moreover,
The latter relation and the symmetry of the Riesz kernel yield
\[
\int_{\mathbb R^N}\rho^{1/p}\,\phi=\int_{\mathbb R^N}  a_p^{1/p}\chi\,\phi(-\Delta)^{-s}\rho=\int_{\mathbb R^N}a_p^{1/p}\chi\,\rho(-\Delta)^{-s}\phi\qquad\forall \phi\in C^\infty_c(\mathbb R^N),
\]
which we write in terms of $u$ as
\begin{equation}\label{ext}
\int_{\mathbb R^N} u\phi=\int_{\mathbb R^N} a_p \chi^{p}\,u^p(-\Delta)^{-s}\phi\qquad\forall\phi\in C^\infty_c(\mathbb R^N).
\end{equation}
With an approximation argument, we extend the validity of \eqref{ext} to test functions of the form $\phi=\Ds \zeta$, $\zeta\in C^\infty(\mathbb R^N)$, and we get
\[
\int_{\mathbb R^N}u\Ds\zeta=\int_{\mathbb R^N} a_p\chi^p\,u^p\zeta
\]
for any $\zeta\in C^\infty_c(\mathbb R^n)$ (note that by Lemma 5.4. in \cite{Ao-Gonzalez-Hyder-Wei}, $u^p\in L^{1}_{-s}$). Therefore, $u$ is a distributional solution to \eqref{groundstatem}, it is radially decreasing and vanishing at infinity (as $\rho$), hence it is a ground state.
\end{proof}

\begin{proposition}\label{0}
Let $1\le p<\tfrac{N+2s}{N-2s}$ and $\mathcal{K}>0$.
Let $u$ be a ground state to \eqref{groundstatem}, where $a_p=(p+1)^{-p}$.   Then, $\rho:=a_p(\chi u-\mathcal{K})_+^p$ is a radial steady state  according to {\rm Definition \ref{def:steady}} with $m=\tfrac{p+1}{p}$.
\end{proposition}
\begin{proof}
 By Proposition \ref{regularity},  $\Ds u=a_p(\chi u-\mathcal{K})_+^p$ holds pointwise in $\mathbb R^N$. On the other hand, the Riesz potential of $\rho$ is  a bounded  radially decreasing vanishing function and the symmetry if the Riesz kernel entails $\int_{\mathbb R^N}\rho(-\Delta)^{-s}\phi=\int_{\mathbb R^N}\phi(-\Delta)^{-s}\rho$ for any $\phi\in C^{\infty}_c(\mathbb R^N)$, thus $(-\Delta)^{-s}\rho=u$ in the sense of distributions and  a.e. in $\mathbb R^N$. Hence, \eqref{eq:steady} holds.
\end{proof}

\begin{proposition}\label{shorty}
Let $p\geq\tfrac{N+2s}{N-2s}$.
Let $u$ be a ground state to $\Ds u=a_p\chi^p u^p$, where $a_p=(p+1)^{-p}$.   Then, $\rho:=a_p\chi^p u^p$ is a radial steady state  according to {\rm Definition \ref{def:steady}} with $m=\tfrac{p+1}{p}$.
\end{proposition}
\begin{proof}
By Corollary \ref{coro}, $u$ is smooth. Moreover, thanks to Lemma 5.4 in \cite{Ao-Gonzalez-Hyder-Wei}, $\rho= a_p\chi^p u^p$ satisfies $\int_{\mathbb R^N}\rho(x)(1+|x|^{N-2s})^{-1}\,dx<+\infty$ so that the Riesz potential of $\rho$ is pointwise well-defined. Therefore,  $\Ds u=a_p\chi u^p$ and $\rho^{\/p}=a_p^{1/p}\chi (-\Delta)^{-s}\rho$ are both pointwise equalities among smooth functions.
\end{proof}

\begin{remark}\rm
If $m>m_{c}$, then Lemma \ref{lemma:basic} implies that any minimizer $\rho$ to the energy functional $\mathcal{F}$ defined in \eqref{functional} is a steady state in the sense of Definition \ref{def:steady}. Moreover, putting $\rho=(-\Delta)^{s}u$, $m=1+1/p$ and taking advantage of Proposition \ref{0} and Proposition \ref{shorty}, by applying Proposition \ref{existence} we get   existence of radial steady states in the fair competition regime $m=m_{c}$ and in the aggregation dominated regime $m\in (2N/(N+2s),m_{c})$, whereas Proposition \ref{criticalcase} and Proposition \ref{prop:C=0} yield existence of radial steady states in the range $m\in(1,\tfrac{2N}{N+2s}]$.
\end{remark}

\begin{remark}\label{existencebis}
Since the result from Lemma \ref{lemma:basic} holds for any $m>m_c$, by applying Proposition \ref{onetoone} we obtain an alternative existence result for \eqref{mainequation}
in the regime $1\le p<\tfrac{N}{N-2s}$ (and in fact also in the regime $0<p<1$, where the definition of ground state is the same and Proposition \ref{onetoone} holds true with the same proof).
\end{remark}

\begin{remark}
It is worth making an interesting remark for the regime $p>(N+2s)/(N-2s)$. Indeed, if $\rho$ is any radial steady state, due to the correspondence exploited in Proposition \ref{0} we have $\mathcal{K}=0$ by Proposition \ref{$C=0$} and the asymptotics \eqref{asymptotsuper} gives $\rho\not\in L^{1}(\mathbb R^N)$, \emph{i.e.} radial steady states have no finite mass in the supercritical regime. This feature agrees with the local case $s=1$ as explained in \cite[Theorem 4.8]{BL} thus making our definition of steady state coherent.

\end{remark}

As regards to the \emph{regularity} of the steady states in the diffusion dominated regime $m>m_c$, it is dictated by the existence of another critical exponent

\begin{equation*}
  m^*:=
  \begin{cases}
  \dfrac{2-2s}{1-2s}\,  \qquad &\text{if} \quad N\geq 1 \quad \text{and} \quad s \in (0,1/2)\, , \\
  + \, \infty &\text{if} \quad N\geq 2 \quad \text{and} \quad s \in [1/2,N/2)\, .
  \end{cases}
 \end{equation*}
Indeed, we have the following result, which is  given in \cite{CHMV}.
\begin{theorem}\label{thm:regmin}\cite[Theorem 8]{CHMV}.
 Let $s \in (0,1)$. If $m>m_c$ and let  $\rho$ is a radial steady state of equation \eqref{eq:KS}.
Then
 \begin{enumerate}
 \item
 if $s\in (1/2,1)$ we have $(-\Delta)^{-s}\rho \in {W}^{1,\infty}(\mathbb{R}^N)$, $ \rho^{m-1}\in W^{1,\infty}(\mathbb{R}^N)$ and $\rho \in C^{0,\alpha}(\R^N)$ with $\alpha=\min\{1,\tfrac{1}{m-1}\}$.
 \item if $s\in (0,1/2]$ we have two subcases:
 \begin{enumerate}
 \item[(i)] if $m\leq 2$ or $2<m<m^*$ the same conclusion of case (1) holds;
 \item[(ii)] if $m\geq m^*$, then  $(-\Delta)^{-s}\rho\in C^{0,\gamma}\left(\R^N\right)$ for any $\gamma<(2s(m-1))/(m-2)$ and $\rho\in C^{0,\alpha}\left(\R^N\right)$ for any $\alpha<2s/(m-2)$.

 \end{enumerate}

\end{enumerate}

\end{theorem}

\begin{remark}
In case $m\in(m_c,2]$,  so that $p\ge 1$, the H\"older regularity of radial steady states can be further improved according to the fact that $\rho=\Ds u$ where $u$ is a ground state (thanks to Proposition \ref{onetoone}), whose H\"older regularity properties are discussed  Proposition \ref{regularity}.
\end{remark}

\begin{remark} 
In the cases (1) and (2)-(i) of the previous Theorem, we have that Definition \ref{def:steady} easily implies that $ \rho^{m} \in W_{\rm loc}^{1,2}\left(\R^N\right)$, $\nabla (-\Delta)^{-s}\rho\in L^1_{\rm loc}\left(\R^N\right)$, and
 it satisfies
 \begin{equation}\label{eq:sstates}
  \nabla  \rho^m= \chi \rho \nabla (-\Delta)^{-s}\rho
 \end{equation}
in the sense of distributions in $\R^N$. Moreover $\rho\in C^{0,\alpha}$ for $\alpha>1-2s$. This is actually the definition of steady state given in \cite{CHVY}, \cite{CHMV}. In particular, one of the main results of \cite{CHVY} shows that densities satisfying \eqref{eq:sstates} must be necessary radially decreasing (up to translation). On the other hand, if we have a steady state defined in the latter sense for the same ranges of $m$ and $s$, \cite[Proposition 1]{CHMV} and \cite[Theorem 3]{CHMV} imply that $\rho$ is radial and satisfies  \eqref{eq:steady}. In the case $m\geq m^{\ast}$ and $s\in (0,1/2]$, which is not covered by our theory, the case (2)-(i) of Theorem \ref{thm:regmin} (satisfied by the minimizers of $\mathcal{F}$ in that range), suggests that a weaker definition of general steady state would be in order and  radial symmetry of all steady states is still an open question.
\end{remark}

We proceed to the proof of uniqueness of radial steady states in the different regimes.
We start with the case $m\in (m_c,2]$.
If $\rho$ is a steady state of mass $M$ of \eqref{eq:KS}, in the sense of Definition \ref{def:steady}, then  $u:=(-\Delta)^{-s}\rho$ is a ground state to  equation \eqref{groundstatem}, thanks to Proposition \ref{onetoone}.
As a direct consequence of Theorem \ref{maintheorem} we obtain the uniqueness of radial steady states, as summarized in the next four propositions.

\begin{proposition}[Diffusion-dominated regime]\label{Uniqueness steady}
Let  $m\in (m_c,2]$. Then for any mass $M>0$ there is a unique radial steady state of mass $M$ in the sense of {\rm Definition \ref{def:steady}}.
\end{proposition}
\begin{proof}
We assume w.l.o.g. that $\chi=1$.
Let $m\in (m_c,2]$. In this case, the existence of steady states of mass $M$ is given in \cite[Theorem 5]{CHMV} by means of minimization of the free energy functional $\mathcal{F}$, see Lemma \ref{lemma:basic}. We put as always $p=\tfrac1{m-1}$, so that $1\le p <\tfrac{N}{N-2s}$. Assume that $\rho_{1},\,\rho_{2}$ are two radial steady states of mass $M$, with respective Lagrange multipliers $\mathcal{K}_1, \mathcal{K}_2$.  Let $u_{i}=(-\Delta)^{-s}\rho_{i}$. By Proposition \ref{onetoone}, $u_i$ is the ground state to \eqref{groundstatem} with $\mathcal{K}=\mathcal{K}_i$.
 We observe that the function
\[
v(x):=\tfrac{\mathcal{K}_{2}}{\mathcal{K}_{1}}\,u_{1}\left(\left(\tfrac{\mathcal{K}_2}{\mathcal{K}_1}\right)^{(p-1)/(2s)}\,x\right)
\]
is a ground state to \eqref{groundstatem} with Lagrange multiplier $\mathcal{K}_{2}$, thus by Theorem \ref{maintheorem} we have $u_2\equiv v$, implying
\[
\lim_{|x|\rightarrow+\infty}\frac{u_{2}(|x|)}{|x|^{2s-N}}=\left(\frac{\mathcal{K}_2}{\mathcal{K}_{1}}\right)^{\frac{N-p(N-2s)}{2s}}\lim_{|x|\rightarrow+\infty}
\frac{u_{1}(|x|)}{|x|^{2s-N}}.
\]
But \eqref{asympbehM} shows that the two limits appearing in the above expression are equal to $M$, thus $\mathcal{K}_1=\mathcal{K}_2$. We conclude that $u_1\equiv u_2$, hence $\rho_1\equiv \rho_2$.
\end{proof}

\begin{remark}\label{lastremark}
In the diffusion-dominated regime,
uniqueness of radial steady states of given mass holds true also for $m> 2$, as a consequence of the result in \cite{CCH3}. Indeed, given a radial steady state $\rho$ of mass $M$, from \eqref{eq:steady} we deduce the a.e. identity $\nabla (\rho^m)=\chi\rho\nabla(-\Delta)^{-s}\rho$.  It is shown in \cite{CCH3} that there is only one radially decreasing solution with mass $M$ to the latter equation. In particular, it is the unique minimizer of functional \eqref{functional} over $\mathcal Y_M$.
Moreover, in this way we also deduce the validity of the uniqueness result of Theorem \ref{maintheorem} for $0<p<1$. Indeed,
thanks to the correspondence between \eqref{eq:steady} and \eqref{groundstatem}, by the usual scaling argument of Proposition \ref{Uniqueness steady} we  infer that  two solutions to \eqref{groundstatem} necessarily coincide.
\end{remark}


\begin{proposition}[Fair competition regime] 
Let $m=m_c$. There is a critical mass $M_{c}>0$ such that all the existing radial steady states to \eqref{eq:KS} according to {\rm Definition \ref{def:steady}}  have mass $M_c$. Moreover, they are minimizers of the free energy functional $\mathcal{F}$ over $\mathcal Y_{M_c}$, they are infinitely many  and all of them are dilations of each other.
\end{proposition}
\begin{proof} We assume w.l.o.g. that $\chi=1$. By invoking \cite[Proposition 3.4]{CCH}
there exists a critical mass $M_c>0$ and a radially decreasing minimizer $\bar\rho\in L^\infty(\mathbb R^N)$ of $\mathcal F$ over $\mathcal Y_{M_c}
$ that satisfies \eqref{eq:steady} for a suitable Lagrange multiplier $\bar{\mathcal{K}}>0$, and moreover $\mathcal F(\bar\rho)=0$. It is easily seen, since $m=m_c$ and $\mathcal F(\bar\rho)=0$, that for any $\lambda>0$ the dilation $\bar\rho_{\lambda}(x):=\lambda^{N}\bar\rho(\lambda x)$ is still of mass $M_c$, it satisfies $\mathcal F(\bar\rho_\lambda)=0$ and it is a radial steady state, satisfying in particular
\[
 \rho(x)^{m-1} = \frac{m-1}{m} \left( \chi(-\Delta)^{-s}\rho(x)-\lambda^{N-2s}\bar {\mathcal{K}}\right)_+\, , \qquad
 \mbox{for a.e.} \, x \in \R^N.
\]
 We have therefore a one-parameter family of radial steady states $\{\bar\rho_\lambda\}_{\lambda>0}$, each having mass $M_c$ and each being a minimizer of $\mathcal F$ over $\mathcal Y_{M_C}$. Moreover, letting $\bar u_\lambda:=(-\Delta)^{-s}\bar\rho_\lambda$, Proposition \ref{onetoone} implies that $u_\lambda$ is a ground state to $\Ds \bar u_\lambda=a_p(\bar u_\lambda-\lambda^{N-2s} \bar{\mathcal{K}})_+^p$ with $p=\tfrac{1}{m_c-1}$ and $a_p=(p+1)^{-p}$.

Suppose now that $\rho$ is a radial steady state with Lagrange multiplier $\mathcal{K}$. By Proposition \ref{onetoone}, $u:=(-\Delta)^{-s}\rho$ is a ground state to $\Ds u=a_p(u-\mathcal{K})_+^p$. By Theorem \ref{maintheorem}, we conclude that $u=\bar u_{\lambda_*}$ where $\lambda_*:=(\mathcal{K}/\bar {\mathcal{K}})^{1/(N-2s)}$. Thus $\rho=\bar\rho_{\lambda_*}$ and $\rho$ belongs to the above one-parameter family of radial steady states.
\end{proof}

\begin{proposition}[Subcritical aggregation dominated regime] 
Let  $m\in (2N/(N+2s),m_c)$. For any mass $M>0$ there exists a unique radial steady state of mass $M$ to \eqref{eq:KS} (in the sense of {\rm Definition \ref{def:steady}}).
\end{proposition}
\begin{proof} Assume w.l.o.g. that $\chi=1$.
The proof of the uniqueness is the same as the proof of Proposition \ref{Uniqueness steady}, thus we briefly focus on the existence part. Set $p=1/(m-1)$ and let $u_1$ be a ground state to \eqref{groundstatem} with the choice $\mathcal{K}=1$ of the constant therein, given by Proposition \ref{existence}. Then
$
u_{1,\lambda}(x):=\lambda u_1(\lambda^{\frac{p-1}{2s}}x)
$
solves
\[
(-\Delta)^{s}u_{1,\lambda}=\left(\tfrac{m-1}{m}\right)^p(u_{1,\lambda}-\lambda)_{+}^{p}.
\]
If we set
$
\rho:=(-\Delta)^{s}u_{1,\lambda},
$
the choice
\[
\lambda=\left(\frac{M}{M_1}\right)^{\frac{2s}{2sp-N(p-1)}}
\]
where
\[
M_1=\|(-\Delta)^{s}u_1\|_{L^1(\mathbb R^N)}
\]
ensures that
$
\int_{\mathbb R^N}\rho =M.
$
By Proposition \ref{0}, $\rho$ is a radial steady state with mass $M$.
\end{proof}



\begin{proposition}[Critical and supercritical regimes]
Let $1<m\le\tfrac{2N}{N+2s}$. Then there exists a unique radial steady state $\rho$  (in the sense of {\rm Definition \ref{def:steady}}) such that $\rho(0)=1$.  The family of functions $\{\rho_\lambda\}_{\lambda>0}$, where \begin{equation}\label{scalingsupercsteady}\rho_\lambda(x):=\lambda^{\tfrac{1}{m-1}}\,\rho\left(\lambda^{\tfrac{2-m}{2s(m-1)}}x\right),\end{equation} is the set of all radial steady states.
\end{proposition}
\begin{proof}
Assume w.l.o.g. that $\chi=1$.
Let $m<\tfrac{2N}{N+2s}$.
 The existence of a unique ground state $u$ for the equation $\Ds u=a_pu^p$, where $a_p=(p+1)^{-p}$, such that $u(0)=1$, is guaranteed by Theorem \ref{th:super}.  Then, by Proposition \ref{shorty}, $\rho:=a_pu^p$ is a radial steady state for $m=\tfrac{p+1}{p}$, with $\rho(0)=a_p$.
Proposition \ref{onetoone} entails uniqueness of such radial steady state $\rho$: indeed, if we are given  another radial steady state $\bar \rho$ with  central density $a_p$, by Proposition \ref{onetoone} its Riesz potential $\bar u$ is a ground state to \eqref{groundstatem}, and Proposition \ref{$C=0$} implies $\mathcal{K}=0$, hence $\bar u(0)=1$ and  Theorem \ref{th:super} implies $\bar u=u$, thus $\bar\rho=\rho$. Eventually, it is clear that $u_\lambda(x):=\lambda u(\lambda^{(p-1)/(2s)}x)$ satisfies  $\Ds u_\lambda=a_pu_\lambda^p$ and $u_\lambda(0)=\lambda$, for any $\lambda>0$. By the same reasoning, given $\lambda>0$, $\rho_\lambda=a_pu_\lambda^p$ is the unique steady state whose value at $x=0$ is $a_p\lambda^p$. Eventually, for the case $m=\tfrac{2N}{N+2s}$, we can use Remark \ref{remarkcrit} and the result by Proposition \ref{criticalcase} to check that all the steady states are of the form \eqref{scalingsupercsteady}, where the steady state of unit central density has the explicit form
\[
\rho=c(N,s)^{\frac{2(N+2s)}{N-2s}}\frac{1}{(c(N,s)^{\frac{4}{N-2s}}+|x|^2)^{\frac{N+2s}{2}}}
\]
being $c(N,s)$ the explicit constant appearing in Proposition \ref{criticalcase}.
\end{proof}

\section{Mass scaling properties}\label{mass scaling}

Let $m>m_c:=2-2s/N$.
For $M>0$, we next denote by $\rho_M$  the unique minimizer of $\mathcal F$ over $\mathcal Y_M$, by $\mathcal F_M$ the minimal value, and by $\mathcal K_M$ the associated Lagrange multiplier obtained from \eqref{C_M}. We also let $u_M:=c_{N,s}|\cdot|^{2s-N}\ast\rho_M$.
This section is devoted to the behavior of these quantities as functions of the mass $M$. We stress that uniqueness of minimizers is a consequence of Proposition \ref{Uniqueness steady} if $m\in(m_c,2]$, but it also known for $m>2$, see Remark \ref{lastremark}. Since the results in this section are only based on uniqueness of minimizers, they hold for any $m>m_c$.

\begin{lemma}[Basic estimates]\label{basicestimates} Let $M>0$. Let $\rho_M$ be a minimizer of $\mathcal F$ over $\mathcal Y_M$. Then
\begin{equation}\label{normestimate}
\|\rho_M\|_m^m \le Q_{\chi,s,m,N}\, M^{\frac{(m-2)N+2sm}{(m-2)N+2s}},
\end{equation}
\begin{equation}\label{Festimate}
{\tilde Q}_{\chi,s,m,N}\, M^{\frac{(m-2)N+2sm}{(m-2)N+2s}}\le -\mathcal F(\rho_M)\le \tilde {\tilde Q}_{\chi,s,m,N}\, M^{\frac{(m-2)N+2sm}{(m-2)N+2s}}
\end{equation}
where 
$Q_{\chi,s,m,N}$, $\tilde Q_{\chi,s,N,m}$ and $\tilde{\tilde Q}_{\chi,s,m,N}$ are positive constants depending only on $\chi,s,m,N$.
\end{lemma}
\begin{proof}
By the standard Hardy-Littlewood-Sobolev inequality   \cite[Theorem 4.3]{LL} and interpolation of $L^p$ norms, there exists a constant $C^*_{s,N,m}$ such that
\[
\int_{\mathbb R^N}\int_{\mathbb R^N}|x-y|^{2s-N}\rho_M(x)\rho_M(y)\,dx\,dy\le C^*_{s,N,m}\,M^{2s/N}\,\|\rho_M\|_{m_c}^{m_c}.
\]
By \eqref{scaleidentity}, by the above inequality and by interpolation of $L^p$ norms again we have
\[\begin{aligned}
\int_{\mathbb R^d}\rho_M^m&=\frac{\chi(N-2s)}{2N}\int_{\mathbb R^N}\int_{\mathbb R^N} c_{N,s}|x-y|^{2s-N}\rho_M(x)\rho_M(y)\,dx\,dy
\\&\le \frac{\chi(N-2s)}{2N}\, C^*_{s,N,m}\,c_{N,s} M^{2s/N}\|\rho_M\|_{m_c}^{m_c}\le  Q_{\chi, s,N,m} M^{2s/N+(1-\theta)m_c}\left(\int_{\mathbb R^N}\rho_M^m\right)^{\frac{\theta m_c}{m}},
\end{aligned}\]
where $Q_{\chi,s,N,m}=\frac{\chi(N-2s)}{2N}\, C^*_{s,N,m}c_{N,s}$ and $\theta=\frac{m(m_c-1)}{m_c(m-1)}\in(0,1)$.
Since $m_c=2-2s/N$, \eqref{normestimate} follows. By taking into account \eqref{functional} and \eqref{scaleidentity}, the second estimate in \eqref{Festimate}  follows as well.

In order to prove the first estimate of \eqref{Festimate}, we look for optimal states among characteristic functions $\bar{\rho}_M=
\frac{M}{\omega_N R^N}\bC_{B_R}$ with given total mass
$M$,
where $\omega_N = \pi^{N/2}/\Gamma(1+N/2)$ is the volume of the unit ball in $\mathbb{R}^N$.
We have
\[
  \frac{1}{m-1}\int_{\mathbb R^N} \rho^m =
  \frac{ M^{m}\omega_N^{1-m}}{m-1}R^{(1-m)N}.
\]
Denoting by $J_\nu$  the Bessel function of the first kind of order $\nu\ge-1/2$, from the following formula for the Fourier transform of a
radially symmetric function $F(x) = f(|x|)$,
\[
  \int_{\mathbb{R}^N} F(x)e^{ix\cdot \xi}dx
  = (2\pi)^{N/2} |\xi|^{(2-N)/2}\int_0^\infty f(\eta) J_{(N-2)/2}(\eta |\xi|)\eta^{N/2}d\eta,
\]
letting $\lambda=\tfrac{M}{\omega_N R^N}$ we get
\begin{align*}
  \int_{\mathbb{R}^N} \bar{\rho}_M(x)e^{i\xi \cdot x} dx
  &= \lambda (2\pi)^{N/2} |\xi|^{(2-N)/2}\int_0^R J_{(N-2)/2}(\eta |\xi|)\eta^{N/2}d\eta\cr
  &= \lambda (2\pi)^{N/2} |\xi|^{-N}\int_0^{|\xi|R} J_{(N-2)/2}(\eta)\eta^{N/2}d\eta \cr
  &= \lambda (2\pi)^{N/2} R^{N/2}|\xi|^{-N/2} J_{N/2}(|\xi|R),
\end{align*}
using the fact that $\int z^{\nu+1} J_{\nu}(z) = z^{\nu+1}J_{\nu+1}(z)$.
Therefore, by Plancherel theorem we compute
\[\begin{aligned}
&\frac12\int_{\mathbb R^N}\int_{\mathbb R^N}c_{N,s}|x-y|^{2s-N}\rho(x)\rho(y)\,dx\,dy=
  \frac{1}{2}\int_{\mathbb{R}^N} \bar{\rho}_M (-\Delta)^{-s}\bar{\rho}_M\\&\qquad=
  \frac{1}{2(2\pi)^N} \int_{\mathbb{R}^N} \left|\hat{\bar{\rho}}_M(\xi)\right|^2|\xi|^{-2s}d\xi
 =
  \frac{1}{2}\lambda^2 R^{N}\int_{\mathbb{R}^N} |\xi|^{-2s-N} \big|J_{N/2}(|\xi|R)\big|^2 d\xi 
  \\&\qquad= \frac{1}{2}\lambda^2 R^N\, N\,\omega_N
  \int_0^\infty \eta^{-2s-1} \big| J_{N/2}(\eta R)\big|^2 d\eta =\frac{1}{2}\lambda^2 R^{N+2s}\,N\,\omega_N \int_0^\infty \eta^{-2s-1}
  \big| J_{N/2}(\eta)\big|^2  d\eta
  \\&\qquad= \frac{1}{4\sqrt{\pi}}\lambda^2 R^{N+2s}\,N\,\omega_N \frac{\Gamma(s+\frac{1}{2})\Gamma(\frac{N}{2}-s)}
  {\Gamma(s+1)\Gamma(\frac{N}{2}+s+1)} = \frac{N\,M^2 \Gamma(s+\frac{1}{2})\Gamma(\frac{N}{2}-s)}
  {4\sqrt{\pi}\omega_N \Gamma(s+1)\Gamma(\frac{N}{2}+s+1)} R^{2s-N}.
\end{aligned}\]
Hence,
\[
\mathcal F(\bar \rho_M)= \frac{ M^{m}\omega_N^{1-m}}{m-1}R^{(1-m)N}- \frac{\chi\,N\,M^2 \Gamma(s+\frac{1}{2})\Gamma(\frac{N}{2}-s)}
  {4\sqrt{\pi}\omega_N \Gamma(s+1)\Gamma(\frac{N}{2}+s+1)} R^{2s-N}
\]
and the optimization of $\mathcal F(\bar\rho_M)$ with respect to $R\in(0,+\infty)$ entails the unique solution
\[
  R=\bar R_M: = \left[
    \frac{2\sqrt{\pi}\, \omega_N^{2-m}\,\Gamma(s+1)\Gamma(\frac{N}{2}+s+1)}
    { \chi\,\Gamma(s+\frac{1}{2})\Gamma(\frac{N}{2}-s+1)}\,M^{m-2}
  \right]^{\frac{1}{2s+(m-2)N}}.
\]
A computation shows that the corresponding minimal value is
\[
 \frac{(2-m)N-2s}{(m-1)(N-2s)}\,\left(\frac{2\sqrt{\pi}\,\Gamma(s+1)\Gamma(\frac N2+s+1)}{\chi\,\Gamma(s+\frac12)\Gamma(\frac N2-s+1)}\right)^{\frac{(1-m)N}{(m-2)N+2s}}
\omega_N^{\frac{2s(1-m)}{(m-2)N+2s}}\,M^{\frac{(m-2)N+2sm}{(m-2)N+2s}},
\]
which is negative since $m>m_c$. The first estimate in \eqref{Festimate} is proven.
\end{proof}

\begin{lemma}[Monotonicity] The mapping $M\mapsto-\mathcal F_M/M$ is strictly increasing on $(0,+\infty)$ with $\lim_{M\to0^+} -\mathcal F_M/M=0$ and $\lim_{M\to+\infty}\mathcal F_M/M=+\infty$. The same properties hold for the map $M\mapsto\mathcal K_M$.
\end{lemma}
\begin{proof}
Let $M>0$ and $\delta>1$. We have $\delta\rho_M\in \mathcal Y_{\delta M}$. Let us compute with \eqref{functional} and \eqref{scaleidentity}
\begin{equation*}\label{comp}\begin{aligned}
\mathcal F(\delta \rho_M)&=\frac{\delta^m-\delta^2}{m-1}\,\int_{\mathbb R^N}\rho_M^m+\frac{\delta^2}{m-1}\,\int_{\mathbb R^N}\rho_M^m-\frac{\chi\delta^2}2\int_{\mathbb R^N}\int_{\mathbb R^N}c_{N,s}|x-y|^{2s-N}\rho_M(x)\rho_M(y)\,dx\,dy
\\&=\frac{(\delta^m-\delta^2)(N-2s)}{(2-m)N-2s}\,\mathcal F(\rho_M)+\delta^2\mathcal F(\rho_M).
\end{aligned}\end{equation*}
Therefore,
\[
\mathcal F(\delta\rho_M)-\delta\mathcal F(\rho_M)= \left(\frac{(\delta^m-\delta^2)(N-2s)}{(2-m)N-2s}+\delta^2-\delta\right)\,\mathcal F(\rho_M).
\]
Taking into account that $\mathcal F(\rho_M)<0$ as seen in Lemma \ref{basicestimates}, the above right hand side is negative if and only if
\begin{equation}\label{delta}
\frac{N-2s}{(m-2)N+2s}\,\frac{\delta^2(\delta^{m-2}-1)}{\delta^2-\delta}<1.
\end{equation}
The latter holds true for any $\delta>1$ in case $m\le2$. If $m>2$, notice that $\lim_{\delta\to 1^+}\tfrac{\delta^{m-2}-1}{\delta-1}=m-2$ so that the left hand side in \eqref{delta} goes to $\tfrac{(m-2)(N-2s)}{(m-2)N+2s}$ which is smaller than $1$ (since $m>2$). This implies the existence of $\delta_0>1$ (only depending on $m,N,s$) such that \eqref{delta} holds true for any $\delta\in(1,\delta_0)$.
Hence, we deduce that $\mathcal F(\delta\rho_M)<\delta\mathcal  F(\rho_M)$ for any $\delta\in(1,\delta_0)$.

Thanks to the minimality of $\rho_{\delta M}$ over $\mathcal Y_{\delta M}$ we conclude that for any $M>0$ and $\delta\in(1,\delta_0)$ there holds
\[
\frac{\mathcal F(\rho_{\delta M})}{\delta M}\le\frac{\mathcal F(\delta\rho_M)}{\delta M}<\frac{\mathcal F(\rho_M)}{M},
\]
implying that the map $(0,+\infty)\ni M\mapsto \mathcal F(\rho_M)/M$ is strictly decreasing.
The limit values at $0$ and $+\infty$ are deduced from Lemma \ref{basicestimates} since $\tfrac{(m-2)N+2sm}{(m-2)N+2s}>1$.
On the other hand,  from \eqref{C_M} we deduce the same properties for the  mapping $(0,+\infty)\ni M\mapsto \mathcal K_M$.
\end{proof}

The next theorem improves the above result by showing that the mappings $M\mapsto-\mathcal F_M/M$, $M\mapsto\mathcal K_M$, $M\mapsto\rho_M(0)$ and $M\mapsto u_M(0)$ are increasing diffeomorphisms of $(0,+\infty)$ onto itself. Moreover, we show how the radius of the support of $\rho_M$ varies with $M$.
\begin{theorem}\label{scaletheorem}
There hold
\begin{equation}\label{rhou}
\rho_M(x) =M^{\frac{2s}{(m-2)N+2s}}\,\rho_1\left(M^{\frac{2-m}{(m-2)N+2s}}x\right),\; u_M(x):= M^{\frac{2s(m-1)}{(m-2)N+2s}}\, u_1\left(M^{\frac{2-m}{(m-2)N+2s}}x\right),\; x\in\mathbb R^N.
\end{equation}
Moreover,
\begin{equation}\label{fc}
\mathcal K_M:= M^{\frac{2s(m-1)}{(m-2)N+2s}}\,\mathcal K_1,\qquad \mathcal F_M= M^{\frac{2s(m-1)}{(m-2)N+2s}}\,\mathcal F_1,
\end{equation}
and denoting by $R_M$ the radius of the support of $\rho_M$, we have
\begin{equation}\label{rm}
R_M=u_M^{-1}(\mathcal K_M)= M^{\frac{m-2}{(m-2)N+2s}}\,u_1^{-1}(\mathcal K_1).
\end{equation}
In particular, the mapping $M\mapsto R_M$ is increasing if $m>2$, decreasing if $m<2$ and constant if $m=2$.
Eventually, if $M_n $ converge to $M>0$ as $n\to+\infty$, we have $\rho_{M_n}\to\rho_M$ and $u_{M_n}\to u_M$ uniformly on $\mathbb R^N$.
\end{theorem}

\begin{proof}
Letting $u_1:=c_{N,s}|\cdot|^{2s-N}\ast\rho_1$, by Theorem \ref{maintheorem} $u_1$ is the unique ground state for
\[
(-\Delta)^su_1=a_m\,(\chi\,u_1-\mathcal K_1)_+^{\frac1{m-1}},\qquad a_m:=\left(\frac{m-1}{m}\right)^{\frac1{m-1}}.
\]
For $\lambda>0$, the usual scaling \begin{equation}\label{usualscaling}u_{1,\lambda}(x):=\lambda u_1\left(\lambda^{\frac{2-m}{2s(m-1)}}x\right),\quad x\in\mathbb R^N,\end{equation} produces a solution to the same equation with different Lagrange multiplier, i.e,
\begin{equation}\label{intermezzo}
(-\Delta)^s u_{1,\lambda}=a_m\,(\chi\,u_{1,\lambda}-\lambda\mathcal K_1)_+^{\frac1{m-1}}.
\end{equation}
Letting $\rho_{1,\lambda}:=(-\Delta)^s u_{1,\lambda} $ we obtain therefore  \begin{equation}\label{rhoscaling}\rho_{1,\lambda}(x)=\lambda^{\frac1{m-1}}\rho_1\left(\lambda^{\frac{2-m}{2s(m-1)}}x\right),\quad x\in\mathbb R^N.\end{equation} The mass of $\rho_{1,\lambda}$ is computed by a change of variables and it is
\[
\int_{\mathbb R^N}\rho_{1,\lambda}=\lambda^{\frac1{m-1}}\int_{\mathbb R^N}\rho_1\left(\lambda^{\frac{2-m}{2s(m-1)}}x\right)\,dx=\lambda^{\frac{(m-2)N+2s}{2s(m-1)}}.
\]
Notice that the latter exponent is positive since $m>m_c$. For any $M>0$,
we define $\lambda_M:=M^{\frac{2s(m-1)}{(m-2)N+2s}}$ so that $\rho_{1,\lambda_M}\in\mathcal Y_M$.
We see from \eqref{intermezzo} that $u_{1,\lambda_M}$ solves
\begin{equation}\label{um}
(-\Delta)^s u=a_m\,(\chi\,u-\mathcal K_M)_+^{\frac1{m-1}},
\end{equation}
where $\mathcal K_M:=\lambda_M\mathcal K_1$.
 In particular, by the uniqueness result of Theorem \ref{maintheorem}, $u_{1,\lambda_ M}$ is the unique ground state for such equation. Hence,  $\rho_{1,\lambda_M}$ coincides  with the  unique minimizer $\rho_M$ of $\mathcal F$ over $\mathcal Y_M$,  the corresponding Lagrange multiplier is $\mathcal K_M$, and $u_{1,\lambda_M}\equiv u_M$. 	\eqref{rhou} is therefore obtained from \eqref{rhoscaling} and \eqref{usualscaling}.
 Notice that $\mathcal K_M=\lambda_M\mathcal K_1$ is the first relation in \eqref{fc}, while the second one follows from \eqref{C_M}. 
Eventually,  since $u_m$ is radially (strictly) decreasing and since it solves \eqref{um} we deduce $u_M(R_M)=\mathcal K_M$ and \eqref{rm} follows from $\mathcal K_M=\lambda_M\mathcal K_1$. The last statement is a direct consequence of \eqref{rhou}, since $\rho_1$ and $u_1$ are continuous and vanishing at infinity.
\end{proof}

\section{Numerical approximation of the fractional plasma equation} 
\label{numerical}

We denote by $\bar u$ the unique ground state in $\dot H^s(\mathbb R^N)$
to $(-\Delta)^s u=(u-1)_+^p$, which is provided by Theorem \ref{maintheorem} for $p\ge 1$.
Some relevant quantities associated to $\bar u$ are
\begin{itemize}
	\item the fractional Laplacian $\bar \rho:=(-\Delta)^s u=(\bar u-1)_+^p$, defined on $\mathbb R^N$,
	\item the mass $\bar M$ of $\bar \rho$, i.e., $\bar M:=\int_{\mathbb R^N}\rho$,
	\item the radius $\bar R$ of the support of $\bar \rho$, i.e., $\bar R:=\sup \{|x|: \bar\rho(x)>0 \}$,
	\item the central density $\bar u(0)=\|\bar u\|_{L^\infty(\mathbb R^N)}>1$,
	\item the oscillation of $\bar u$ inside $B_{\bar R}$, i.e., $\bar u(0)-\bar u(Rx/|x|)=\bar u(0)-1$.
\end{itemize}
Let us now consider the following two-parameter family $\{u_{\mathcal{C}, \delta}\}_{\mathcal{C}>0,\,\delta>0}$ of functions
\[
u_{\mathcal{C},\delta}(x):=\mathcal{C}\,\bar u(\delta x),\qquad x\in\mathbb R^N.
\]
We immediately obtain, by using Theorem \ref{maintheorem}, that $u_{\mathcal{C},\delta}$
is the unique ground state in $\dot H^s(\mathbb R^n)$ to
\begin{equation}\label{muC}
(-\Delta)^s u=a\,(u-\mathcal{C})_+^p,\qquad\mbox{where $\;\;a=\mathcal{C}^{1-p}\,\delta^{2s}$}.
\end{equation}
This shows that the family $\{u_{\mathcal{C},\delta}\}$ can be equivalently parameterized by the couple of positive numbers $(\mathcal{C},a)$ or, in case $p>1$, by the couple $(\delta,a)$.
After having defined
\[
\rho_{\mathcal{C},\delta}(x):=(-\Delta)^s u_{\mathcal{C},\delta}(x)=\mathcal{C}\delta^{2s}\bar\rho(\delta x),\qquad x\in\mathbb R^N,
\]
we can reason as done in Theorem \ref{scaletheorem}
and identify an element of the family $\{u_{\mathcal{C},\delta}\}$ by prescribing the mass of $\rho_{\mathcal{C},\delta}$ along with $\delta$ or $\mathcal{C}$, since a direct computation shows that
\[
M_{\mathcal{C},\delta}:=\int_{\mathbb R^N}\rho_{\mathcal{C},\delta}(x)\,dx=\mathcal{C}\delta^{2s-N}\bar M.
\]
More generally, denoting by $R_{\mathcal{C},\delta}$ the  radius of the support of $\rho_{\mathcal{C},\delta}$,
we have the following relations
\begin{equation}\label{list}
\begin{aligned}
R_{\mathcal{C},\delta}&=\frac {\bar R}\delta,\qquad u_{\mathcal{C},\delta}(R_{\mathcal{C},\delta})=\mathcal{C},\\ u_{\mathcal{C},\delta}(0)&=\mathcal{C} \bar u(0), \qquad
u_{\mathcal{C},\delta}(0)-u_{\mathcal{C},\delta}(R_{\mathcal{C},\delta})=\mathcal{C}\,(\bar u(0)-1).\end{aligned}
\end{equation}
This shows that it is possible to uniquely identify any element of the family  $\{u_{\mathcal{C},\delta}\}$
by prescribing, for instance,  the radius of the support and either  the parameter $a$ appearing in  \eqref{muC} or the oscillation inside the support. The latter choice will be useful in the numerical approximations of our interest in this section;
indeed, it is more convenient to work with numerical solutions whose fractional Laplacian is supported in the unit ball, and whose oscillation inside the unit ball is prescribed, while continuously depending on the rest of the parameters.
We note moreover that in the family $\{u_{\mathcal{C},\delta}\}$ each of the following quantities uniquely identifies the other two: the oscillation inside the support, the central density, the Lagrange multiplier $\mathcal C$. In the special case $p=1$, we see from  \eqref{muC} and  \eqref{list} that the value of $a$ uniquely identifies the radius of the support. In particular if $a$ is given, the radius of the support does not depend on the mass, a property that we have already obtained in Theorem \ref{scaletheorem}.\\[0.5pt]

 Now, solutions to $(-\Delta)^s u_{\mathcal{C},\delta} = a (u_{\mathcal{C},\delta}-\mathcal{C})_+^p$ will be approximated
numerically with the normalisation $R_{\mathcal{C},\delta}=1$ and $u_{\mathcal{C},\delta}(0)-u_{\mathcal{C},\delta}(R_{\mathcal{C},\delta})
=1$. The key formula is the following expressions (see Appendix~\ref{app:riesz} for the derivation) for
the Riesz potential of  the weight Jacobi polynomial $(1-|x|^2)^{-s}P_n^{(-s,N/2-1)}  (2|x|^2-1)$, that is,
\begin{align}
\begin{cases}
\lambda_n P_n^{(-s,N/2-1)}  (2|x|^2-1), \quad & |x| < 1,\\
\lambda_n \mu_n|x|^{-d-2s-2n}
{ }_2F_1\left( 1-s+n,
\frac{N}{2}+n-s; 1+2n+\frac{N}{2}-s;|x|^{-2}\right), & |x|>1,
\end{cases}
\label{eq:riesz}
\end{align}
where
\[ \lambda_n = \frac{2^{-2s}\Gamma(1+n-s)\Gamma(N/2-s+n)}{n!\Gamma(N/2+n)}\ \mbox{ and }  \
\mu_n = \frac{\sin s\pi}{\pi} B\left(1+n-s,\frac{N}{2}+n\right),
\]
with the Beta function $B(p,q)=\Gamma(p)\Gamma(q)/\Gamma(p+q)$.
Since $\rho_{\mathcal{C},\delta}= (-\Delta)^s u_{\mathcal{C},\delta}$ is assumed to be supported on the unit ball, $\rho_{\mathcal{C},\delta}$ can be expanded in terms of the series
\[
\rho_{\mathcal{C},\delta}(x) = (1-|x|^2)^{-s} \sum_{n=0}^\infty c_n P_n^{(-s,N/2-1)}  (2|x|^2-1)
\]
on the unit ball with some unknown coefficients $\{c_n\}_{n=0}^\infty$, then
from \eqref{eq:riesz}, the solution $u_{\mathcal{C},\delta}$ on the unit ball can be expressed as
\begin{equation}\label{eq:uexpn}
u_{\mathcal{C},\delta}(x) = (-\Delta)^{-s} \rho_{\mathcal{C},\delta}(x) = \sum_{n=0}^\infty \lambda_n c_n P_n^{(-s,N/2-1)}  (2|x|^2-1).
\end{equation}
Therefore the governing equation $(-\Delta)^s u_{\mathcal{C},\delta} = a (u_{\mathcal{C},\delta}-\mathcal{C})_+^p$ becomes
\begin{multline} \label{eq:jacobiPDE}
(1-|x|^2)^{-s} \sum_{n=0}^\infty c_n P_n^{(-s,N/2-1)}  (2|x|^2-1) =  \cr
a\left(\sum_{n=0}^\infty \lambda_n c_n \Big(P_n^{(-s,N/2-1)}  (2|x|^2-1) -  P_n^{(-s,N/2-1)}(1)\Big)\right)^p,
\end{multline}
where the relation $\mathcal{C} = u_{\mathcal{C},\delta}(x)|_{|x|=1} = \sum_{n=0}^\infty \lambda_n c_n P_n^{(-s,N/2-1)}  (1)$
is applied. Using the orthogonality condition for Jacobi polynomials
\[ \int_{\{|x|\leq 1\}} (1-|x|^2)^{-s}P_n^{(-s,N/2-1)}  (2|x|^2-1)P_k^{(-s,N/2-1)}  (2|x|^2-1) \,dx = 0, \qquad n \neq k,
\]
Equation \eqref{eq:jacobiPDE} can be further reduced (with the change of variable $t=2|x|^2-1$) to a system of
algebraic equations for the coefficients $\{c_k\}$, that is,
\[
c_k = \frac{a}{2^sQ_k} \int_{-1}^1
(1+t)^{N/2-1}\left(\sum_{n=0}^\infty \lambda_n c_n \Big( P_n^{(-s,N/2-1)}  (t) - P_n^{(-s,N/2-1)}(1)\Big)\right)^p
P_k^{(-s,N/2-1)}(t)\,
dt,
\]
for $k = 0,1,\ldots$, where $Q_k$ is the normalisation constant defined by
\[
Q_k = \int_{-1}^1 (1-t)^{-s}(1+t)^{N/2-1}\big[P_k^{(-s,N/2-1)}(t)\big]^2dt
=\frac{2^{N/2-s}}{2k+N/2-s}\frac{\Gamma(k+1-s)\Gamma(k+N/2)}{k!\Gamma(k+N/2-s)}.
\]
In practice, the series is truncated with finite number of coefficients
$\mathbf{c} = (c_0,c_1,\ldots,c_K)$, leading to a system of $K+2$ algebraic equations
for the variables $\tilde{\mathbf{c}} =  (\mathbf{c},a) = (c_0,c_1,\ldots,c_K,a)$:
the first $K+1$ equations take the form $c_k = F_k(\tilde{\mathbf{c}})$ with $F_k(\tilde{\mathbf{c}})$ defined as
\[
\frac{a}{2^sQ_k} \int_{-1}^1
(1+t)^{N/2-1}\left(\sum_{n=0}^K \lambda_n c_n \Big( P_n^{(-s,N/2-1)}  (t) - P_n^{(-s,N/2-1)}(1)\Big)\right)^p
P_k^{(-s,N/2-1)}(t)\,
dt,
\]
for $k=0,1,\ldots,K$, and the last equation is given by the normalisation $1 = u_{\mathcal{C},\delta}(0)-u_{\mathcal{C},\delta}(R_{\mathcal{C},\delta})$, i.e.,
\begin{equation}\label{eq:normal}
1 = \sum_{n=0}^K \lambda_n c_n \Big( P_n^{(-s,N/2-1)}  (-1) - P_n^{(-s,N/2-1)}(1)\Big).
\end{equation}
This system of $K+2$ equation is denoted as $\mathbf{G}(\tilde{\mathbf{c}}) = 0$.

For $p=1$, the system of algebraic equation can be treated as an eigenvalue problem,
where $a$ plays the role of an eigenvalue and the entries of the associated eigenvectors
are exactly the expansion coefficients $\mathbf{c} = (c_0,c_1,\cdots,c_K)$. Therefore, the solution can be obtained
by standard numerical linear algebra packages.
For  $p \in (0,1)$, the coefficients can be obtained using the
fixed point iteration $\mathbf{c}^{(m+1)} = \mathbf{F}(\mathbf{c}^{m},a^{(m)})$ by taking the
first $K+1$ equations in $\mathbf{G}(\tilde{\mathbf{c}}) = 0$, and $a^{(m+1)}$ is chosen such that
the normalisation in Eq.~\eqref{eq:normal} is satisfied.
This fixed point iteration converges for a wide range of initial conditions, for instance with $c_k=0$ for all $k$
except that $c_1 > 0$.
The numerical solutions in one dimension with $p=0.5$ and various values of $s$ are shown in Figure \ref{fig:1dsmallp},
together with its fractional Laplacian $\rho$. For fixed $s=1/2$, the numerical solutions for different values of $p$
in two dimension are shown in Figure \ref{fig:2dsmallp}, where $\rho$ is converging to a characteristic function.
\begin{figure}[htp]
	\begin{center}
		\includegraphics[totalheight=0.26\textheight]{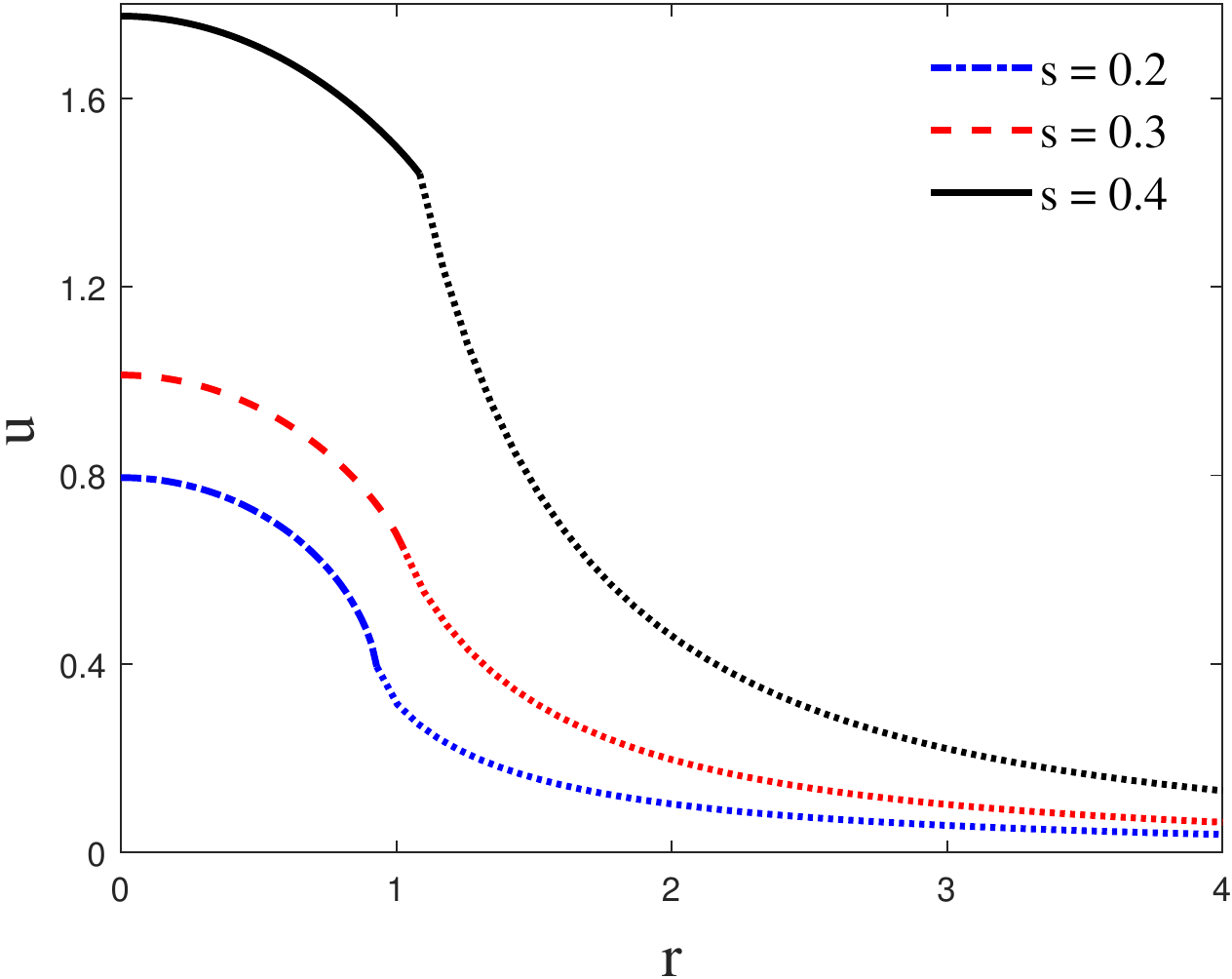}$~~~$		
		\includegraphics[totalheight=0.26\textheight]{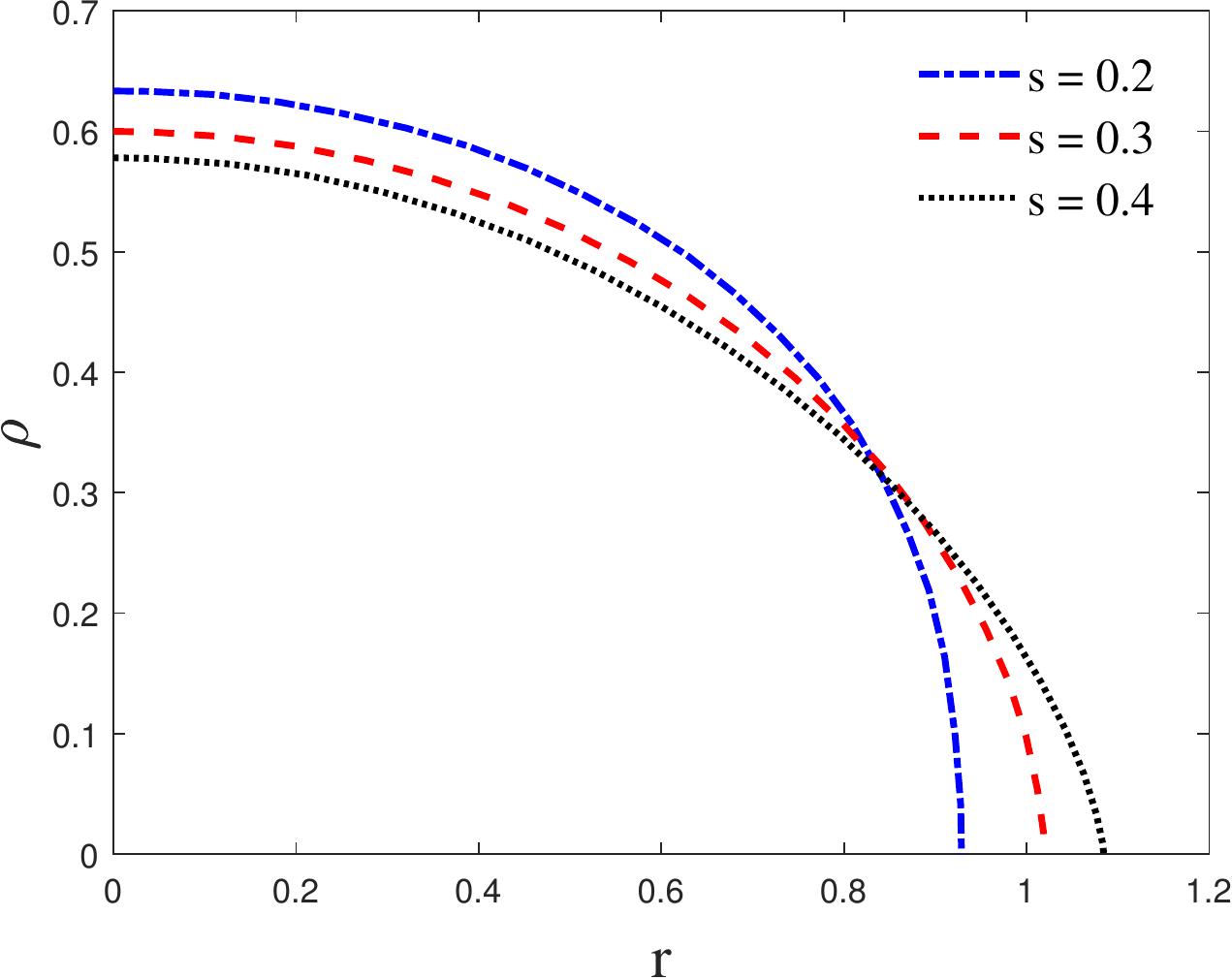}
	\end{center}
	\caption{The solution $u$ (left) and its fractional Laplacian $\rho$ (right) for $p=0.5$ and different $s$
		in one dimension.}
	\label{fig:1dsmallp}
\end{figure}

\begin{figure}[htp]
	\begin{center}
		\includegraphics[totalheight=0.25\textheight]{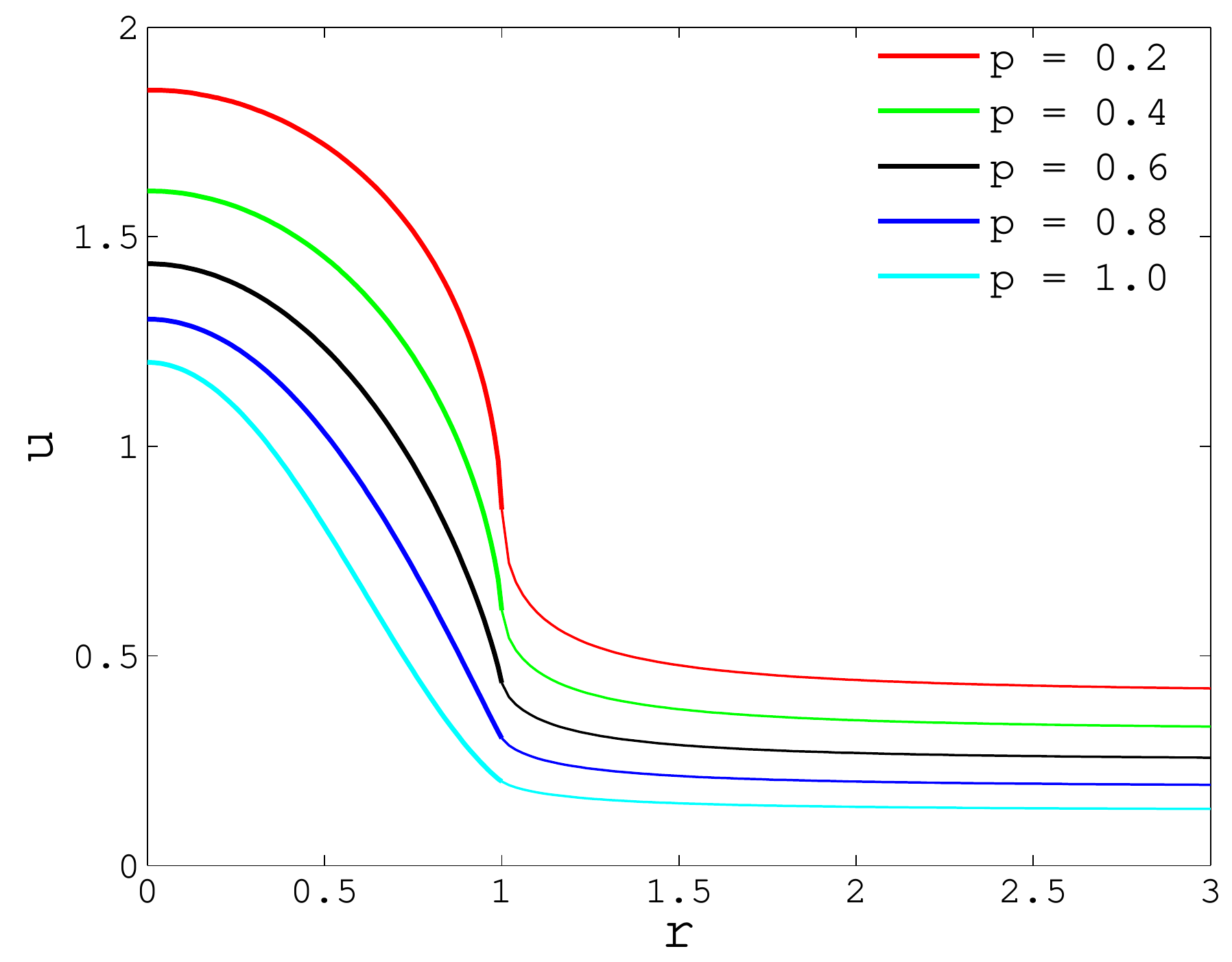}$~~~$		
		\includegraphics[totalheight=0.25\textheight]{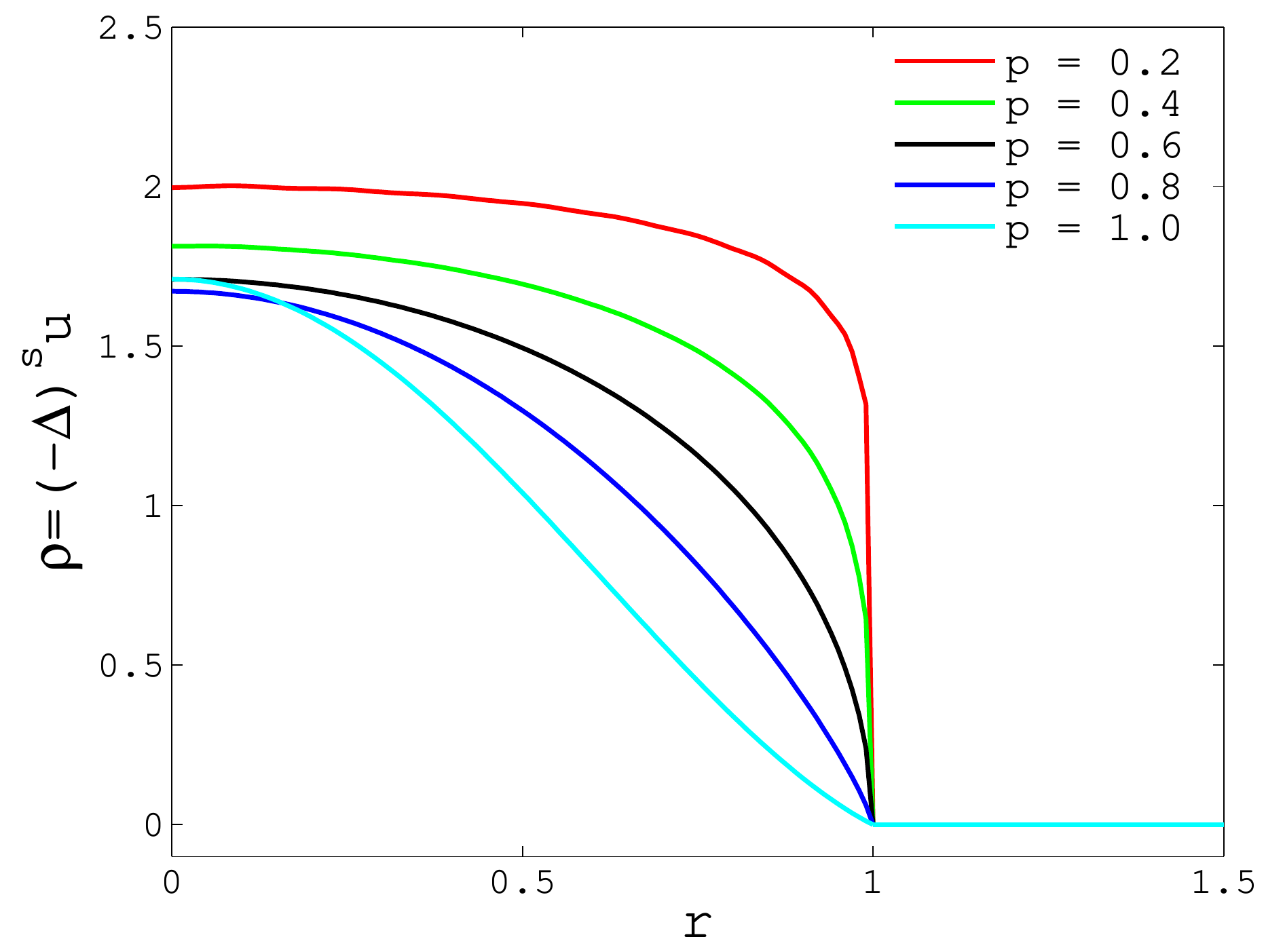}
	\end{center}
	\caption{The solution $u$ (left) and its fractional Laplacian $\rho$ (right) for different $p$ not larger than $1$
		in two dimensions with $s=0.5$.}
	\label{fig:2dsmallp}
\end{figure}

However, for the case $p>1$ of our interest, the above fixed point iteration
does not seem to converge, and  Newton's method for nonlinear equations is applied, i.e.,
\[
\tilde{\mathbf{c}}^{(m+1)} = \tilde{\mathbf{c}}^{(m)}
- \big({\boldsymbol{\partial}}{\mathbf{G}}(\tilde{\mathbf{c}}^{(m)})\big)^{-1}
{\mathbf{G}}(\tilde{\mathbf{c}}^{(m)}),
\]
where ${\boldsymbol{\partial}}{\mathbf{G}}(\tilde{\mathbf{c}}^{(m)})$
is the Jacobian matrix of ${\mathbf{G}}(\tilde{\mathbf{c}})$. Since a good initial guess
is essential for the convergence of the Newton's method, the solution at any $p > 1$
is continued from the case $p = 1$:  the numerical solution is computed first for $p = 1$, and
then the exponents $p$ is increased by a small amount, until the desired exponent is reached.
Numerical experiments indicate that the algorithm always converges with an  increment of $\Delta p =0.1$.
The radial solutions in dimension two for $p = 1.0, 1.2, 1.4$ and $1.6$ (with $s = 0.5$) is shown in Figure~\ref{fig:diffp}.

\begin{figure}[htp]
	\begin{center}
		\includegraphics[totalheight=0.26\textheight]{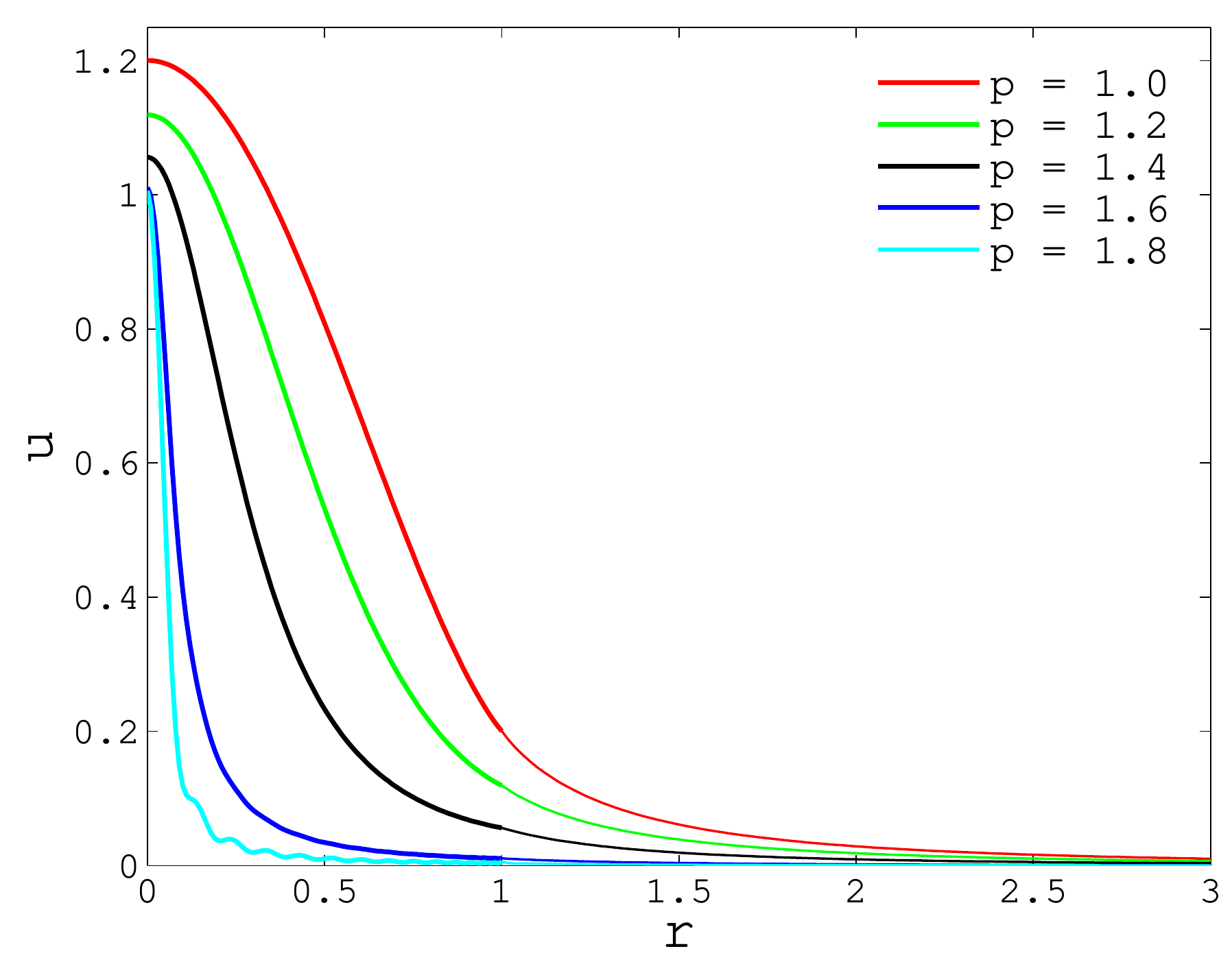}$~~~$		
		\includegraphics[totalheight=0.26\textheight]{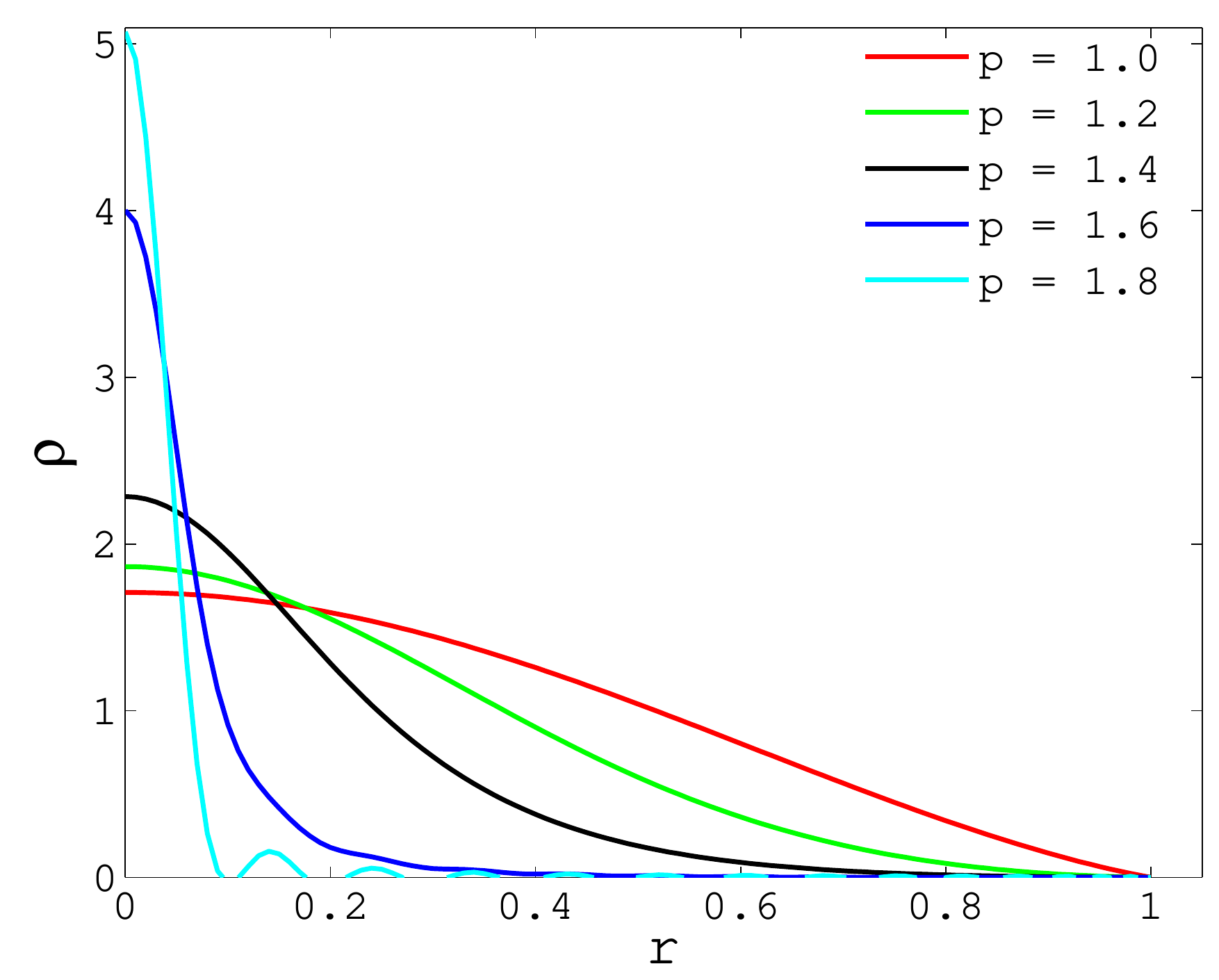}
	\end{center}
	\caption{The radial solution $u$ (left) and its fractional Laplacian $\rho$ (right) in dimension two  for $p = 1.0, 1.2, 1.4, 1.6$ and $1.8$, with $s = 0.5$.}
		\label{fig:diffp}
\end{figure}

However, as the values of $p$ approach its upper limit $(N+2s)/(N-2s)$, the solution $u$
becomes more concentrated near the origin, and the coefficients $c_n$ in \eqref{eq:uexpn}
decays slower and slower, as shown in Figure \ref{fig:coefdecay} in dimension two for
different exponents $p$ with $s=1/2$.
As a result, the number of coefficients $K$ has to be larger and larger in order to resolve the solution faithfully,
otherwise artificial oscillation could appear as for the case $p=1.8$ in Figure \ref{fig:diffp}, with
the slow decay of the coefficients as the exponent $p$ increases shown in Figure~\ref{fig:coefdecay}.


\begin{figure}[htp]
	\begin{center}
		\includegraphics[totalheight=0.28\textheight]{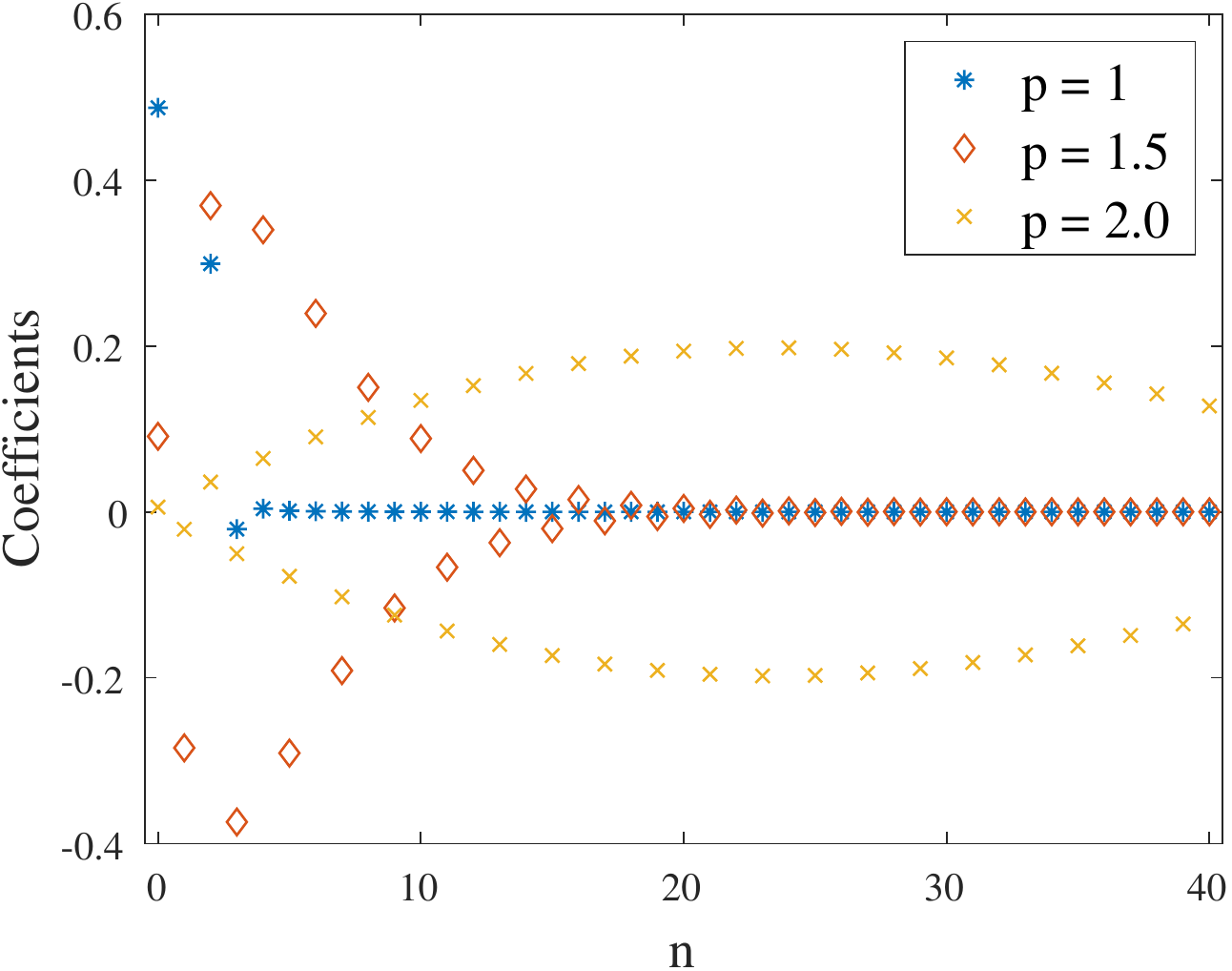}
	\end{center}
	\caption{The decay of the coefficients $c_n$ of $\rho$ for $p=1$, $p=1.5$ and $p=2.0$ respectively, for
		$s=0.5$ in  dimension two.}
		\label{fig:coefdecay}
\end{figure}

\appendix
\section{Riesz potential of the weighted Jacobi polynomials}
\label{app:riesz}
Here we give a brief derivation of the expressions in \eqref{eq:riesz} about the Riesz potential of the weighted Jacobi polynomials
$(1-|x|^2)^{-s}P_n^{(-s,N/2-1)}  (2|x|^2-1)$ restricted on the unit ball.
This relation can be established essentially by reversing the sign of $s$
as for the fractional Laplacian of $(1-|x|^2)^{s}P_n^{(s,N/2-1)}  (2|x|^2-1)$ in~\cite[Theorem 3]{MR3640641},
so that the Riesz potential can be represented as the inverse Mellin transform
\begin{equation}\label{eq:mellin}
\frac{(-1)^n2^{-2s}\Gamma(1+n-s)}{n!}\frac{1}{2\pi i}\int_{\mathscr{C}} \frac{\Gamma(\tau)\Gamma(\frac{N}{2}-s+n-\tau)}{\Gamma(\frac{N}{2}-\tau)\Gamma(1+n+\tau)}|x|^{-2\tau}d\tau,
\end{equation}
where $\mathscr{C}$ is a contour from $\sigma-i\infty$ to $\sigma+i\infty$ with $0 < \sigma < N/2-s+n$. If $|x|<1$, the
contour integral is reduced to the sum of residues around the poles of $\Gamma(\tau)$, leading to
\begin{multline*}
\frac{(-1)^n2^{-2s}\Gamma(1+n-s)}{n!}\sum_{k=0}^n \frac{(-1)^k}{k!}\frac{\Gamma(N/2-s+n+k)}{\Gamma(N/2+k)\Gamma(1+n-k)}|x|^{2k} \cr
=\lambda_n(-1)^n
\frac{\Gamma(N/2+n)}{n!\Gamma(N/2)} {}_2F_1(-n,N/2+n-s;N/2;|x|^2)
=\lambda_nP_n^{(-s,N/2-1)}(2|x|^2-1)
\end{multline*}
using the equivalent definition $P_n^{(a,b)}(z) = (-1)^n \frac{\Gamma(1+b+n)}{n!\Gamma(1+b)} {}_2F_1(-n,1+a+b+n;1+b; (1+z)/2)$
for Jacobi polynomials. For $|x|>1$, the contour integral \eqref{eq:mellin} is evaluated by summing the
residues around the poles of $\Gamma(\frac{N}{2}-s+n-\tau)$, leading to
\begin{align*}
&\quad \frac{(-1)^n2^{-2s}\Gamma(1+n-s)}{n!} \sum_{k=0}^\infty \frac{(-1)^k}{k!}\frac{\Gamma(N/2+n+k-s)}
{
	\Gamma(s-n-k)\Gamma(N/2+2n+1-s+k)
}|x|^{-N-2n-2k+2s} \cr
&= \frac{2^{-2s}\Gamma(1+n-s)\sin \pi s}{n!\pi} |x|^{-N-2n+2s}
\sum_{k=0}^\infty \frac{\Gamma(N/2+n+k-s) \Gamma(1+n-s+k)}
{
	\Gamma(N/2+2n+1-s+k)k!
}|x|^{-2k}  \cr
&= \lambda_n \mu_n |x|^{-N-2n+2s}{ }_2F_1\left( 1-s+n,
\frac{N}{2}+n-s; 1+2n+\frac{N}{2}-s;|x|^{-2}\right).
\end{align*}

\noindent\textbf{Acknowledgements.}
The authors wish to warmly thank Y. Sire, X. Cabr\'{e}, J. Dolbeault, N. Ikoma and L. Montoro for the fruitful discussions and valuable suggestions.
This work has been partially supported by GNAMPA of the Italian INdAM (National Institute of High Mathematics). H.C. has received funding from the European Research Council under the Grant Agreement No 721675.
M.d.M. Gonz\'alez is supported by the Spanish government grant  MTM2017-85757-P.  E.M. acknowledges support from the MIUR-PRIN  project  No 2017TEXA3H and from the INdAM-GNAMPA 2019 project   {\it ``Trasporto ottimo per dinamiche con interazione''}. B.V. acknowledges support from the `Programma
triennale della Ricerca dell'Universit\`{a} degli Studi di Napoli ``Parthenope'' - Sostegno alla ricerca individuale 2015-2017'' and the INDAM-GNAMPA 2019 project   {\it ``Trasporto ottimo per dinamiche con interazione''}.

\bibliographystyle{plain}

\Addresses

 \end{document}